\documentclass{amsart}
\usepackage{amsfonts}

\newtheorem{theorem}{Theorem}
\theoremstyle{plain}

\newtheorem{conjecture}{Conjecture}
\newtheorem{corollary}{Corollary}

\newtheorem{definition}{Definition}
\newtheorem{example}{Example}

\newtheorem{lemma}{Lemma}

\newtheorem{remark}{Remark}

\numberwithin{equation}{section}
\input{tcilatex}

\begin{document}
\title[EIP on $S(n,m)$]{The Edge-Isoperimetric Problem on Sierpinski
Graphs:\ Final Resolution}
\author{L. H. Harper}
\address{Department of Mathematics\\
University of California-Riverside\\
Riverside, CA 92521}
\email{harper@math.ucr.edu}
\date{February 15, 2018}

\begin{abstract}
This paper completes the project started in \cite{Har16}; to solve the
edge-isoperimetric problem on the (generalized and extended) Sierpinski
graph, S(n,m). We prove that initial segments of lexicographic order are
solutions of the $EIP$ for all $n,m$.
\end{abstract}

\maketitle

\section{ Introduction}

\subsection{Motivation}

William, Rajasingh, Rajan \& Shanthakumari \cite{W-R-R-S} proposed the
Sierpinski pyramid graph, $S\left[ n,4\right] $ as the connection graph of a
multiprocessor computer. In their conclusion they suggest studying $S\left[
n,4\right] $ for its "message routing and broadcasting" properties. This
paper is the last of three (\cite{Har16}, \cite{H-H-L}) following up on that
suggestion. The Edge-Isoperimetric Problem ($EIP$, see \cite{Har04}) is of
interest for connection graphs of multiprocessor computers because it has
implications for message routing and broadcasting.

\subsection{Basic Definitions}

\begin{definition}
An \textit{ordinary} \textit{graph, }$G=\left( V,E\right) $ consists of a
set $V$, of \textit{vertices} and a set $E\subseteq \binom{V}{2}=\left\{
\left\{ v,w\right\} :v,w\in V,v\neq w\right\} $, of pairs of vertices called 
\textit{edges}.
\end{definition}

\begin{example}
\bigskip $K_{m}$, \textit{the complete graph on }$m$ vertices has $%
V_{K_{m}}=\left\{ 0,1,2,...,m-1\right\} $ and $E_{K_{m}}=\binom{V_{K_{m}}}{2}
$.
\end{example}

\begin{example}
The (disjunctive) product, $K_{m}\times K_{m}\times ...\times
K_{m}=K_{m}^{n} $ is called the Hamming graph. $V_{K_{m}^{n}}=\left\{
0,1,2,...,m-1\right\} ^{n}$. Two vertices ($n$-tuples of vertices of $K_{m}$%
) have an edge between them if they differ in exactly one coordinate (%
\textit{i.e. }are at Hamming distance $1$). Note that $K_{2}^{n}=Q_{n}$, the
graph of the $n$-dimensionsal cube.
\end{example}

\subsubsection{The Edge-Isoperimetric Problem}

The Edge-Isoperimetric Problem (EIP) is a combinatorial analog of the
classical isoperimetric problem: Given a graph, $G=\left( V,E\right) $ and $%
S\subseteq V$, 
\begin{equation*}
\Theta \left( S\right) =\left\{ \left\{ v,w\right\} \in E:v\in S\text{ }\&%
\text{ }w\notin S\right\}
\end{equation*}%
is called the \textit{edge-boundary of }$S$. Then the $EIP$ is to calculate $%
\left\vert \Theta \right\vert \left( G;\ell \right) =\min \left\{ \left\vert
\Theta \left( S\right) \right\vert :S\subseteq V_{G}\&\left\vert
S\right\vert =\ell \right\} $ for every integer $\ell $, $0\leq \ell \leq
\left\vert V\right\vert $, and identify sets that achieve the minimum.

\begin{example}
For $K_{m}$, \textit{the complete graph on }$m$ vertices, any $S\subseteq
V_{K_{m}}$ with $\left\vert S\right\vert =\ell $ has $\left\vert \Theta
\left( S\right) \right\vert =\ell \left( m-\ell \right) $. Thus every $\ell $%
-set is a solution of the EIP for $K_{m}$.
\end{example}

\begin{example}
Initial $\mathbf{\ell }$-segments of $V_{K_{m}^{n}}$ in Lexicographic order, 
\begin{equation*}
\left\{ 0^{n},0^{n-1}1,...,\mathbf{\ell }_{1}\ell _{2}...\ell _{m}\right\}
\end{equation*}%
where $\ell =1+\sum_{i=1}^{m}\ell _{i}m^{n-i}$, are solutions of the EIP on $%
K_{m}^{n}$ (proved for $m=2$ by the author in 1962 and for $m>2$ by John
Lindsay in 1963). See \cite{Har04} for more on the $EIP$.
\end{example}

\begin{definition}
The function $\left\vert \Theta \right\vert \left( G;\ell \right) $ is
called the (edge-)isoperimetric profile of $G$.
\end{definition}

\begin{example}
The isoperimetric profile of $K_{m}$ is $\left\vert \Theta \right\vert
\left( K_{m};\ell \right) =\ell \left( m-\ell \right) $.
\end{example}

\subsection{The Sierpinski Graph}

The \textit{generalized \& extended Sierpinski graph, }$S(n,m)$, $n\geq 1$, $%
m\geq 2$, was defined in 1944 by Scorer, Grundy and Smith \cite{S-G-S}. They
showed that $S(n,3)$ is the graph of the Tower of Hanoi puzzle with $n$
discs (see \cite{H-K-M-P} for its colorful history). The following
representation of $S(n,m)$, implicit in the Scorer-Grundy-Smith paper, was
made explicit by Klav\v{z}ar and Milutinovi\'{c} in 1997 \cite{K-M}: $%
V_{S(n,m)}=\left\{ 0,1,...,m-1\right\} ^{n}$. For $\left\{ u,v\right\} \in 
\binom{V}{2}$, $\left\{ u,v\right\} \in E_{S(n,m)}$ iff $\exists h\in
\left\{ 1,2,...,n\right\} $ such that following 3 conditions hold:

\begin{enumerate}
\item $u_{i}=v_{i}$ for $i=1,2,...h-1$;

\item $u_{h}\neq v_{h}$; and

\item $u_{j}=v_{h}$ and $v_{j}=u_{h}$ for $j=h+1,...,n$.
\end{enumerate}

\begin{remark}
\bigskip The vertices of $\ S(n,m)$ are of degree $m$ except for those of
the form $i^{n}$ (called "corner vertices"), that are of degree $m-1$.
\end{remark}

\begin{conjecture}
Initial $\mathbf{\ell }$-segments of $V_{S(n,m)}$ in Lexicographic order,
are solutions of the EIP on $S(n,m).$
\end{conjecture}

\begin{remark}
These are the same sets that solve the $EIP$ on the Hamming graph, $%
K_{m}^{n} $. $V_{S(n,m)}=\left\{ 0,1,...,m-1\right\} ^{n}=V_{K_{m}^{n}}$ but
the edge-sets for $K_{m}^{n}$ and $S(n,m)$ are quite different. It is
possible to embed $S(n,m)$ into $K_{m}^{n}$ (see \cite{Har15'}), but then
the representation of $\ $edges in $S(n,m)$ is not the same. Graph theory is
about properties of graphs (chromatic number, clique number, hamiltonicity, 
\textit{etc.) that are independent of the representation of the graph
(invariant under isomorphisms). The standard families of graphs (such as }$%
Q_{n},C_{n},K_{m}$) have just one representation. $S(n,m)$ is unusual in
having three different representations that are not obviously isomorphic.
That the Klav\v{z}ar-Milutinovi\'{c} representation of $S(n,m)$ is the right
one for this paper is evidenced by Conjecture 1.

When graphs arise in applications, $V$ is often some set of structures (such
as positions in the Tower of Hanoi puzzle) and $E$ some set of pairs of
those structures determined by a symmetric relation (a legal move from
position $v$ to position $w$). It is surprisingly difficult to determine
whether two such graphs are isomorphic (See Wikipedia entry , "Graph
isomorphism problem").
\end{remark}

Given $s,t\in \mathbb{N}$, $s+t\leq m$, let 
\begin{eqnarray*}
I &=&\left\{ 0,1,...,s-1\right\} , \\
J &=&\left\{ s,s+1,...,s+t-1\right\} , \\
K &=&\left\{ s+t,s+t+1,...,m-1\right\} ,
\end{eqnarray*}%
so $\left\vert I\right\vert =s,$ $\left\vert J\right\vert =t$ $\&$ $%
\left\vert K\right\vert =m-s-t$. Consider $S_{s,t}(n,m)$ to be the graph $%
S(n,m)$ decorated with "exterior" edges attached to the corner vertices. If $%
i\in I+K$ then $\left\{ v_{i},i^{n}\right\} \in E_{S_{s,t}(n,m)}$ and when
computing $\left\vert \Theta _{s,t}\left( S\right) \right\vert $, $%
S\subseteq \left\{ 0,1,...,m-1\right\} ^{n}$, we consider $v_{i}$ to be a
member of $S$ if $i\in I$ but to be in the complement of $S$ if $i\in K$.
Vertices $j^{n}$ for $j\in J$ are stilll regarded as corner vertices, not
incident to an "exterior" edge. We call $S_{s,t}(n,m)$ a \textit{decorated
Sierpinski graph}.

In \cite{Har16} the following conjecture was stated:

\begin{conjecture}
$\forall \ell ,$ $0\leq \ell \leq m^{n},$ $\left\vert \Theta
_{s,t}\right\vert \left( S_{s,t}(n,m);\ell \right) =$ $\left\vert \Theta
_{s,t}\left( Lex^{-1}\left( n,m;\ell \right) \right) \right\vert $ where $%
Lex^{-1}\left( n,m;\ell \right) =Lex^{-1}(\left\{ 1,2,...,\ell \right\} )$
in $S_{s,t}(n,m)$.
\end{conjecture}

The original goal of \cite{Har16} was to prove Conjecture 1, which
corresponds to the case $s=0$, $t=m$ ($S_{0,m}(n,m)=S(n,m)$, the unaugmented
Sierpinski graph). The extensions of $S(n,m)$ to $S_{s,t}(n,m)$ and of
Conjecture 1 to Conjecture 2 were made to facilitate the definition of
compression, a Steiner operation for $EIP$ based on self-similarity.

In \cite{Har16} Conjecture 2 was proved for $m=2$ and $\forall n$ (which is
trivial) and for $m=3$ and $\forall n$ (which is not trivial). In a followup
paper (\cite{H-H-L}) we set out to prove it for all $\forall n,\forall m$.
We succeeded in advancing the theory, but were only able to verify the
ultimate sufficient condition (Conjecture 3 below) with the aid of a
computer, showing that Conjecture 2, and therefore Conjecture 1, holds for $%
n,m$ such that $n+m\leq 16$. In this paper we prove Conjecture 2 and
therefore Conjecture 1.

We shall repeat relevant definitions and theorems of \cite{Har16} but not
proofs.

\subsection{A Three StOp Proof}

Our approach to proving Conjecture 1 was modeled on the first solution of
the $EIP$ for $Q_{n}$, the graph of the $n$-dimensional cube (Theorem 1.1 in
the monograph \cite{Har04}). That proof is essentially the author's first
publication (1964) \cite{Har64} as corrected by A. J. Bernstein \cite{Ber}
and streamlined with further study. It used three Steiner operations:
Stabilization (based on reflective symmetry of $Q_{n}$), compression (based
on a direct product decomposition with factors having nested solutions) and
Bernstein's Lemma. It was only included in the monograph to show the roots
of the theory of Steiner operations (StOps) and provide a contrast with the
relatively transparent reproof. The development of "pushouts" for multiple
stabilizations and compressions made those StOps much more powerful.
Bernstein's Lemma was no longer needed to solve the $EIP$ on the graph of
the $n$-cube and the original 3-StOp proof was reduced to two StOps. The
theory (pushouts in particular) also applied to many related graphs (such as
the Hamming graph, $K_{m}^{n}=K_{m}\times K_{m}\times ...\times K_{m}$).
However, the Sierpinski graph, $S(n,m)$, has little symmetry compared to the
Hamming graph, so stabilization is comparitively ineffectual. Also, $S(n,m)$%
, is not factorable as a product ($K_{m}^{n}$ has many factorizations) so
compression, as defined in \cite{Har04} does not apply at all. $S(n,m)$ is
self-similar though, a disjoint union of $m$ copies of $S(n-1,m)$ with edges
between their corner vertices. This suggests the possibility of extending
the definition of the crucial operation of compression to $S(n,m)$. In
looking back over the literature of the $EIP$ we realized that the original
proof had treated the graph of the $n$-cube, $Q_{n}=K_{2}^{n}$, as a
self-similar graph, so the original 3-StOp proof might be extended to $%
S(n,m) $. However, it would only work if Bernstein's Lemma could be
extended. Bernstein's Lemma states that the isoperimetric profile of $Q_{n}$
is subadditive. That led to a third Steiner operation that we call
"subadditivation". All three StOps, stabilization, compression and
subaddivation, had to be extensively modified in \cite{Har16}, \cite{H-H-L}
to work for $S(n,m)$. Ultimately, the complexity of proving the required
subadditivity (Conjecture 3 below) for the isoperimetric profile of $S(n,m)$%
, $m>3$, stymied those efforts.

\subsection{New Definitions}

There are two important parameters in the analysis of lexicographic order on 
$V_{S(n,m)}$, $k_{n,m}(\ell )$ and $q_{n,m}\left( \ell \right) $:

\begin{definition}
For $0\leq \ell \leq m^{n},$ $k_{n,m}(\ell )$ is\textit{\ }the number of
subgraphs of the form $\left\{ i\right\} \times S(n-1,m)$ in $Lex^{-1}\left(
n,m;\ell \right) $.
\end{definition}

\begin{lemma}
$k_{n,m}(\ell )=\left\lfloor \frac{\ell }{m^{n-1}}\right\rfloor $
\end{lemma}

\begin{remark}
$k_{n,m}$ is monotone increasing and $0\leq k_{n,m}(\ell )\leq m$.
\end{remark}

Let $C_{n,m}=\left\{ i^{n}:i=0,1,...,m-1\right\} $ (the set of corner
vertices of $S(n,m))$.

\begin{remark}
$Lex\left( i^{n}\right) =1+\sum_{j=1}^{n}im^{n-j}$ $=1+i\frac{m^{n}-1}{m-1}$
gives the numbers, $1$, $1+\frac{m^{n}-1}{m-1}$, $1+2\frac{m^{n}-1}{m-1}$%
,..., $m^{n}$, assigned to those corner vertices in lexicographic order.
\end{remark}

\begin{definition}
For $0\leq \ell \leq m^{n},$ $q_{n,m}(\ell )=\left\vert Lex^{-1}\left(
m,n;\ell \right) \cap C_{n,m}\right\vert $, the number of corner vertices
with $Lex\left( v\right) \leq \ell $.
\end{definition}

\begin{remark}
$q_{n,m}$ is monotone increasing and $0\leq q_{n,m}(\ell )\leq m$. Also, $%
k_{n,m}(\ell )\leq q_{n,m}(\ell )\leq k_{n,m}(\ell )+1$ with $%
k_{n,m}(0)=0=q_{n,m}(0)$ and $k_{n,m}(m^{n})=m=q_{n,m}(m^{n})$. Lastly, 
\begin{eqnarray*}
q_{n,m}(m^{n}-\ell ) &=&\left\vert C_{n,m}-Lex^{-1}\left( n,m;\ell \right)
\cap C_{n,m}\right\vert \\
&=&m-q_{n,m}(\ell ).
\end{eqnarray*}
\end{remark}

\begin{lemma}
For all $m\geq 2$, $q_{0,m}(\ell )=0$ for $\ell =0,1$. For all $n\geq 1$ and 
$0\leq \ell \leq m^{n}$, $q_{n,m}(\ell )=\ 1+\left\lfloor \frac{\left( \ell
-1\right) \left( m-1\right) }{m^{n}-1}\right\rfloor $.
\end{lemma}

\begin{proof}
Note that the interval of integers from $1$ to $m^{n}-1$ may be divided into 
$m-1$ intervals of size $1+m+m^{2}+...+m^{n-1}$. The $k^{th}$ such interval
consists of the numbers $Lex\left( \left( k-1\right) ^{n}\right) =1+\left(
k-1\right) \frac{m^{n}-1}{m-1},$ $2+\left( k-1\right) \frac{m^{n}-1}{m-1}%
,...,\frac{m^{n}-1}{m-1}+$ $\left( k-1\right) \frac{m^{n}-1}{m-1}=k\frac{%
m^{n}-1}{m-1}=Lex\left( k^{n}\right) -1$. Thus $q_{n,m}(\ell )=k+1$ iff $%
Lex\left( k^{n}\right) =1+k\frac{m^{n}-1}{m-1}\leq \ell <1+\left( k+1\right) 
\frac{m^{n}-1}{m-1}=Lex\left( \left( k+1\right) ^{n}\right) $, \textit{i.e. }%
$k\leq \frac{\left( \ell -1\right) \left( m-1\right) }{m^{n}-1}<k+1$.
Therefore $k=\left\lfloor \frac{\left( \ell -1\right) \left( m-1\right) }{%
m^{n}-1}\right\rfloor $ so $q_{n,m}(\ell )=k+1=\ 1+\left\lfloor \frac{\left(
\ell -1\right) \left( m-1\right) }{m^{n}-1}\right\rfloor $.
\end{proof}

\begin{lemma}
For $n\geq 0$, $m\geq 2$ and $0\leq \ell \leq m^{n}$, the values of $%
\left\vert \Theta (Lex^{-1}\left( n,m;\ell \right) )\right\vert $ are
generated from the initial condition, 
\begin{equation*}
\left\vert \Theta (Lex^{-1}\left( 0,m;\ell \right) )\right\vert =\left\{ 
\begin{tabular}{ll}
$0$ & if $\ell =0$ \\ 
$0$ & if $\ell =1$%
\end{tabular}%
\right. \text{,}
\end{equation*}%
by the recurrence, for $n>0,$%
\begin{eqnarray*}
&&\left\vert \Theta (Lex^{-1}\left( n,m;\ell \right) )\right\vert \\
&=&k_{n,m}(\ell )\left( m-k_{n,m}(\ell )\right) +\left\vert \Theta
(Lex^{-1}\left( n-1,m;\ell ^{\prime }\right) )\right\vert \\
&&+\left\{ 
\begin{tabular}{ll}
$-q_{n-1,m}\left( \ell ^{\prime }\right) $ & if $q_{n-1,m}\left( \ell
^{\prime }\right) \leq k_{n,m}(\ell )$ \\ 
$q_{n-1,m}\left( \ell ^{\prime }\right) -2k_{n,m}(\ell )$ & if $%
q_{n-1,m}\left( \ell ^{\prime }\right) >k_{n,m}(\ell )$%
\end{tabular}%
\right. \text{,}
\end{eqnarray*}%
where $\ell ^{\prime }=\ell -k_{n,m}(\ell )m^{n-1}$, so $0<\ell ^{\prime
}\leq m^{n-1}$.

\begin{proof}
The graph $S(0,m)$ has one vertex and no edges so 
\begin{equation*}
\left\vert \Theta (Lex^{-1}\left( 0,m;\ell \right) )\right\vert =\left\{ 
\begin{tabular}{ll}
$0$ & if $\ell =0$ \\ 
$0$ & if $\ell =1$%
\end{tabular}%
\right. \text{.}
\end{equation*}

The edges, $\left\{ v,w\right\} $ of $S(n,m)$ cut by $Lex^{-1}\left( \left[
1,\ell \right] \right) $ have

\begin{enumerate}
\item $Lex\left( v\right) \leq k_{n,m}(\ell )m^{n-1}$ and $Lex\left(
w\right) >k_{n,m}(\ell )m^{n-1}$, or

\item $k_{n,m}(\ell )m^{n-1}<Lex\left( v\right) <Lex\left( w\right) \leq
\left( k_{n,m}(\ell )+1\right) m^{n-1}$, or

\item $k_{n,m}(\ell )m^{n}<Lex\left( v\right) \leq \left( k_{n,m}(\ell
)+1\right) m^{n-1}$ and

\begin{enumerate}
\item $0\,<Lex\left( w\right) \leq k_{n,m}(\ell )m^{n-1}$ or

\item $\left( k_{n,m}(\ell )+1\right) m^{n-1}<Lex\left( w\right) \leq m^{n}$.
\end{enumerate}
\end{enumerate}

The three cases correspond to the three terms of the identity. Note that the
edges which should not be counted in term 1, are cancelled out in term 3a.
\end{proof}
\end{lemma}

\begin{definition}
For $0\leq \ell _{b}\leq \ell _{a}\leq m^{n}$, let%
\begin{equation*}
\sigma _{n,m}\left( \ell _{a},\ell _{b}\right) =\left\{ 
\begin{tabular}{ll}
$q_{n,m}(\ell _{b})+\left( q_{n,m}(\ell _{a}+\ell _{b})-q_{n,m}(\ell
_{a})\right) $ & if $\ell _{a}+\ell _{b}<m^{n}$ \\ 
$q_{n,m}(\ell _{b})-q_{n,m}(\ell _{a}+\ell _{b}-m^{n})+\left( m-q_{n,m}(\ell
_{a})\right) $ & if $\ell _{a}+\ell _{b}\geq m^{n}$%
\end{tabular}%
\right. .
\end{equation*}
\end{definition}

\begin{remark}
\bigskip If $\ell _{a}+\ell _{b}=m^{n}$, both formulas for $\sigma
_{n,m}\left( \ell _{a},\ell _{b}\right) $ give the same value, $q_{n,m}(\ell
_{b})+m-q_{n,m}(\ell _{a})$.
\end{remark}

In \cite{Har16}, to prove Conjecture 1 for $m=3$, we extended all three of
the Steiner operations, stabilization, compression and subadditivation, that
had sufficed to prove that initial segments of Lex order on $K_{2^{n}}$, the
graph of the $n$-cube, are solutions of the EIP. The extensions for
stabilization and compression given in \cite{Har16} hold for all $m$, but
that for subadditivation only holds for $S(n,3)$. The results necessary for
this extension were presented in the following lemmas:

\begin{lemma}
(Theorem 7 of \cite{Har16}) If $3^{n}>\ell _{a}\geq \ell _{b}>0$, then 
\begin{equation*}
\left\vert \Theta \left( Lex^{-1}\left( n,3;\ell _{a}+\ell _{b}\right)
\right) \right\vert +1\leq \left\vert \Theta \left( Lex^{-1}\left( n,3;\ell
_{a}\right) \right) \right\vert +\left\vert \Theta \left( Lex^{-1}\left(
n,3;\ell _{b}\right) \right) \right\vert .
\end{equation*}
\end{lemma}

\begin{lemma}
(Lemma 10 of \cite{Har16})If $3^{n}/2>\ell _{a}\geq \ell _{b}>0$ $\&$ $\ell
_{a}+\ell _{b}>3^{n}/2$, then 
\begin{equation*}
\left\vert \Theta \left( Lex^{-1}\left( n,3;\ell _{a}+\ell _{b}\right)
\right) \right\vert +2\leq \left\vert \Theta \left( Lex^{-1}\left( n,3;\ell
_{a}\right) \right) \right\vert +\left\vert \Theta \left( Lex^{-1}\left(
n,3;\ell _{b}\right) \right) \right\vert .
\end{equation*}
\end{lemma}

In each case the inequality for subadditivity had to be strengthened by an
additive term. The following statement generalizes those results for
arbitrary $m$:

\begin{conjecture}
$\forall n,m,\ell _{a},\ell _{b}\in \mathbb{N}$ such that $m^{n}\geq \ell
_{a}\geq \ell _{b}>0$, if $\ell _{a}+\ell _{b}\leq m^{n}$ then 
\begin{eqnarray*}
&&\left\vert \Theta (Lex^{-1}\left( n,m;\ell _{a}+\ell _{b}\right)
)\right\vert +\sigma _{n,m}\left( \ell _{a},\ell _{b}\right) \\
&\leq &\left\vert \Theta (Lex^{-1}\left( n,m;\ell _{a}\right) )\right\vert
+\left\vert \Theta (Lex^{-1}\left( n,m;\ell _{b}\right) )\right\vert ,
\end{eqnarray*}
\end{conjecture}

and if $\ell _{a}+\ell _{b}\geq m^{n}$ then 
\begin{eqnarray*}
&&\left\vert \Theta (Lex^{-1})\left( n,m;\ell _{a}+\ell _{b}-m^{n}\right)
\right\vert +\sigma _{n,m}\left( \ell _{a},\ell _{b}\right) \\
&\leq &\left\vert \Theta (Lex^{-1})\left( n,m;\ell _{a}\right) \right\vert
+\left\vert \Theta (Lex^{-1})\left( n,m;\ell _{b}\right) \right\vert .
\end{eqnarray*}%
If this condition holds we say that $\left\vert \Theta (Lex^{-1}\left(
n,m;\ell \right) )\right\vert $ is \textit{subadditive+}$\sigma $.

The initial justification for Conjecture 3 was that it is simple and
suffices to prove Conjecture 2 (as we show in the next section). After
expending some effort to prove it, we began to question its validity.
However, verifying it for all $m,n$ such that $m+n\leq 16$ by computer
calculation \cite{H-H-L} convinced us that we were on the right track.
Because of its complexity, we defer the proof of Conjecture 3 until after
the next section which justifies the effort.

\section{Another Two-StOp Proof}

\begin{theorem}
Conjecture 3 $\Rightarrow $ Conjecture 2 .

\begin{proof}
This proof follows essentially the same logic as the 3-StOp proof of the
main theorem in \cite{Har16} for the special case $m=3$. However, it has
been simplified by dropping stabilization, which turned out to be
unnecessary. The proof proceeds by induction on $n$ ($\ell ,m,s,t$ being
fixed).

\begin{description}
\item[Initial Step] It is true for $n=1$ since $S_{s,t}(1,m)=K_{m}$ with the
"extra" vertices, $V_{I},V_{K}$, attached and any $\ell $-set of vertices
which takes the members of $I$ first and the members of $K$ last will
minimize $\left\vert \Theta _{s,t}(S)\right\vert $. Since $Lex_{IJK}=Lex$
does that, its initial segments are optimal. Thus Conjecture 2 is true and
the implication is trivial.

\item[Inductive Step] Assume the theorem is true for $n\geq 1$ and that 
\begin{equation*}
S\subseteq V_{S_{s,t}(n+1,m)}=\left\{ 0,1,...,m-1\right\} ^{n+1}\text{,}
\end{equation*}%
with $\left\vert S\right\vert =\ell $, and $S$ minimizes $\left\vert \Theta
_{s,t}(S)\right\vert $ over all such sets. We may also assume that $S$
maximizes $\boldsymbol{\ell }(S)=\left( \ell _{0},\ell _{1},...,\ell
_{m-1}\right) $ with respect to lexicographic order, where $\ell
_{h}=\left\vert S\cap \left( \left\{ h\right\} \times S_{s,t}(n,m)\right)
\right\vert $. The lexicographic maximum that $\boldsymbol{\ell }(S)$ can
take if $\left\vert S\right\vert =\ell $ is 
\begin{equation*}
\left( m^{n}\right) ^{k_{n,m}\left( \ell \right) }\ell ^{\prime
}0^{n-k_{n,m}\left( \ell \right) }=\boldsymbol{\ell }(Lex^{-1}\left(
1,2,...,\ell \right) ).
\end{equation*}
If $S=Lex^{-1}\left( 1,2,...,\ell \right) $ we are done, so assume that $%
S\neq Lex^{-1}\left( 1,2,...,\ell \right) $. We shall use the following two
Steiner operations to reduce any such $S$ to $Lex^{-1}\left( \left\{
1,2,...,\ell \right\} \right) $, $Lex$ being the standard lexicographic
order on $V_{S(n,m)}=\left\{ 0,1,...,m-1\right\} ^{n+1}$):

\begin{enumerate}
\item Apply compression, utilizing the inductive hypothesis. Then we need
only consider $S$ that are compressed,

\item Apply subadditivation (a StOp based on the subadditivity+$\sigma $ of
\linebreak $\left\vert \Theta \left( S(n,m);\ell \right) \right\vert $
(Conjecture 3)) reducing $S$ to $Lex^{-1}(\left\{ 1,2,...,\ell \right\} )$.
\end{enumerate}
\end{description}

\underline{Compression}: From Section 4.4 of \cite{Har16}: Conjecture 1 is
the special case of Conjecture 2 with $s=0,t=m$. The point is that the
optimal order on $S_{s,t}\left( n,m\right) $ is independent of $s,t$ even
though its exterior edges vary with $s,t$. This is what makes compression
work on $S_{s,t}\left( n+1,m\right) $. Since any permutation of $\left\{
0,1,...,m-1\right\} $ induces a symmetry of $S(n+1,m)$, from the point of
view of $\left\{ h\right\} \times S_{s,t}(n,m),$ its exterior edges whose
other ends are in $S$ may be regarded as coming from the previous ranks
(renumbered $0,1,...,s^{\prime }-1$ but maintaining their relative order)
and the exterior edges whose other ends are not in $S$ may be regarded as
going to the succeeding ranks (similarly renumbered $s^{\prime }+t^{\prime
},s^{\prime }+t^{\prime }+1,...,m-1$). The remaining vertices are also
renumbered $s^{\prime },s^{\prime }+1,...,s^{\prime }+t^{\prime }-1$. It is
wrt this renumbering that we define $Lex_{h}$.

\begin{definition}
\bigskip $\forall S\subseteq V_{S_{s,t}(n+1,m)}$, $Comp_{Lex_{h}}\left(
S\right) =S-\left( S\cap V_{\left\{ h\right\} \times S_{s.t}(n,m)}\right)
+Lex_{h}^{-1}\left\{ 1,2,...,\ell _{h}\right\} $.
\end{definition}
\end{proof}
\end{theorem}

\begin{theorem}
(Theorem 3 of \cite{Har16}) $\forall S\subseteq V_{S_{s,t}(n+1,m)}$,
\end{theorem}

\begin{enumerate}
\item \bigskip\ $\left\vert Comp_{Lex_{h}}\left( S\right) \right\vert
=\left\vert S\right\vert $ and

\item \bigskip $\left\vert \Theta _{s,t}\left( Comp_{Lex_{h}}\left( S\right)
\right) \right\vert \leq \left\vert \Theta _{s,t}\left( S\right) \right\vert 
$.
\end{enumerate}

These are the two essential properties of a Steiner operation. Remember that
the vertices in $V_{I}$ are not actually in $S$ and not counted as such,
even though they can contribute to $\left\vert \Theta _{s,t}(S)\right\vert $%
. Similarly for those in $V_{K}$. For each $h\geq $ $0$, we apply $%
Comp_{h\left( \func{mod}m\right) }$ to $Comp_{\left( h-1\right) \func{mod}%
m}\left( Comp_{\left( h-2\right) \func{mod}m}\left( ...\left( Comp_{0}\left(
S\right) \right) \right) \right) .$ Note that the compositions are applied
cyclically ($\func{mod}m)$.

\begin{theorem}
(Theorem 4 of \cite{Har16}).Cyclic compositions of $Comp_{Lex_{h(\func{mod}%
m)}}\left( S\right) $, $h=0,1,...$, will eventually be constant, defining a
nonmonotone Steiner operation, $Comp_{\infty }$, on $S_{s,t}(n+1,m)$.
\end{theorem}

\ $\left\vert Comp_{\infty }\left( S\right) \right\vert =\left\vert
S\right\vert $ and $\left\vert \Theta _{s,t}(Comp_{\infty }\left( S\right)
)\right\vert \leq \left\vert \Theta _{s,t}(S)\right\vert $ so $Comp_{\infty
}\left( S\right) $ will still be optimal. Also, $\boldsymbol{\ell }%
(Comp_{\infty }\left( S\right) )=\boldsymbol{\ell }(S)$ so $Comp_{\infty
}\left( S\right) $ will still maximize $\boldsymbol{\ell }(S)$ over all
optimal $\ell $-sets. See Section 4.4 of \cite{Har16} for proofs of Theorems
2 \& 3.

\begin{proof}
(Continuing now with the proof of Theorem 1)

\underline{Subadditivation}: Subadditivation is a Steiner operation based on
the fact that $\left\vert \Theta \left( S(n,m);\ell \right) \right\vert $ is
subadditive+$\sigma $. We may assume that our $\ell $-set $S,$ which
minimizes $\left\vert \Theta \left( S\right) \right\vert $ over all $%
S\subseteq V_{S_{s,t}(n+1,m)}$ with $\left\vert S\right\vert =\ell $, is
compressed and maximizes $\boldsymbol{\ell }(S)$ over all such sets. If $%
\boldsymbol{\ell }(S)=\left( m^{n}\right) ^{k_{n,m}\left( \ell \right) }\ell
^{\prime }0^{n-k_{n,m}\left( \ell \right) }$ and $S$ is compressed, then $S$
must be $Lex^{-1}\left( \left\{ 1,2,...,\ell \right\} \right) $.\ Let $%
h_{\min }=\min \left\{ h:\ell _{h}<m^{n}\right\} $ and $h_{\max }=\max
\left\{ h:\ell _{h}>0\right\} $. If $S\neq Lex^{-1}\left( \left\{
1,2,...,\ell \right\} \right) $, then $h_{\min }<h_{\max }$. If that is the
case, let $S^{\prime }=S-S\cap \left( \left\{ h_{\min }\right\} \times
S_{s,t}(n,m)\right) -S\cap \left\{ h_{\max }\right\} \times S_{s.t}(n,m)$
and we have $SubAdd(S)=$ 
\begin{equation*}
\left\{ 
\begin{tabular}{l}
$S^{\prime }+\left\{ h_{\min }\right\} \times Lex_{h_{\min }}^{-1}\left(
\ell _{h_{\min }}+\ell _{h_{\max }}\right) $ \\ 
$\ \ \ \ \ \ \ \ \ \ \ \text{if }\ell _{h_{\min }}+\ell _{h_{\max }}\leq
m^{n}$, \\ 
\\ 
$S^{\prime }+\left\{ h_{\min }\right\} \times S_{s,t}(n,m)+\left\{ h_{\max
}\right\} \times Lex_{h_{\max }}^{-1}\left( \ell _{h_{\max }}-\left(
m^{n}-\ell _{h_{\min }}\right) \right) $ \\ 
$\ \ \ \ \ \ \ \ \ \ \text{ if }\ell _{h_{\min }}+\ell _{h_{\max }}>m^{n}$.%
\end{tabular}%
\right.
\end{equation*}%
In either case $\left\vert SubAdd(S)\right\vert =\left\vert S\right\vert
=\ell $, so $SubAdd$ has property 1 of a StOp. To show that it has Property
2, note that the only edges that could contribute to the difference, $%
\left\vert \Theta _{s,t}\left( S\right) \right\vert -\left\vert \Theta
_{s,t}\left( SubAdd(S)\right) \right\vert $ are the internal and external
edges of $\left\{ h_{\min }\right\} \times S_{s,t}(n,m)$ and $\left\{
h_{\max }\right\} \times S_{s,t}(n,m)$. The contribution from any other edge
would be the same in both terms, thereby cancelling. More precisely we
assert that if $\ell _{h_{\min }}+\ell _{h_{\max }}\leq m^{n}$, then%
\begin{eqnarray*}
&&\left\vert \Theta _{s,t}\left( S\right) \right\vert -\left\vert \Theta
_{s,t}\left( SubAdd(S)\right) \right\vert \\
&\geq &\left\vert \Theta \right\vert \left( n,m;\ell _{h_{\min }}\right)
+\left\vert \Theta \right\vert \left( n,m;\ell _{h_{\max }}\right)
-\left\vert \Theta \right\vert \left( n,m;\ell _{h_{\min }}+\ell _{h_{\max
}}\right) \text{, } \\
&&\text{ \ \ \ \ \ \ \ \ \ the difference due to internal edges, } \\
&&-\left( q_{n,m}\left( \ell _{h_{\max }}\right) +(q_{n,m}\left( \ell
_{h_{\max }}+\ell _{h_{\min }}\right) -q_{n,m}\left( \ell _{h_{\min
}}\right) \right) \text{, } \\
&&\text{ \ \ \ \ \ \ \ \ \ the maximum possible decrease due to external
edges,} \\
&=&\left\vert \Theta \right\vert \left( n,m;\ell _{h_{\min }}\right)
+\left\vert \Theta \right\vert \left( n,m;\ell _{h_{\max }}\right) \\
&&-\left( \left\vert \Theta \right\vert \left( n,m;\ell _{h_{\min }}+\ell
_{h_{\max }}\right) +\sigma _{n,m}\left( \ell _{h_{\min }},\ell _{h_{\max
}}\right) \right) \\
&\geq &0\text{ by Conjecture 3.}
\end{eqnarray*}%
The case $\ell _{h_{\min }}+\ell _{h_{\max }}\geq m^{n}$ follows by a dual
arguement.
\end{proof}

\section{Proof of Conjecture 3}

\subsection{Preliminaries}

\begin{lemma}
$\forall n,m,\ell _{a},\ell _{b}\in \mathbb{N}$, $0\leq \ell _{b}\leq \ell
_{a}\leq m^{n}$. Suppose also that $\ell _{a}+\ell _{b}<m^{n}$, then%
\begin{equation*}
k_{n,m}(\ell _{a}+\ell _{b})=\left\{ 
\begin{tabular}{l}
$k_{n,m}(\ell _{a})+k_{n,m}(\ell _{b})$ \\ 
\ \ \ \ \ \ \ \ \ iff\ $\ \ell _{a}^{\prime }+\ell _{b}^{\prime }<m^{n-1}$
\\ 
$k_{n,m}(\ell _{a})+k_{n,m}(\ell _{b})+1$ \\ 
$\ \ \ \ \ \ \ \ \ $iff $\ \ell _{a}^{\prime }+\ell _{b}^{\prime }\geq
m^{n-1}$%
\end{tabular}%
\text{.}\right.
\end{equation*}

\begin{proof}
\ We have 
\begin{eqnarray*}
\ell _{a} &=&k_{n,m}(\ell _{a})m^{n-1}+\ell _{a}^{\prime }\text{ with }0\leq
\ell _{a}^{\prime }<m^{n-1} \\
\ell _{b} &=&k_{n,m}(\ell _{b})m^{n-1}+\ell _{b}^{\prime }\text{ with }0\leq
\ell _{b}^{\prime }<m^{n-1}\text{,}
\end{eqnarray*}%
and 
\begin{eqnarray*}
\ell _{a}+\ell _{b} &=&\left( k_{n,m}(\ell _{a})m^{n-1}+\ell _{a}^{\prime
}\right) +\left( k_{n,m}(\ell _{b})m^{n-1}+\ell _{b}^{\prime }\right) \\
&=&\left( k_{n,m}(\ell _{a})+k_{n,m}(\ell _{b})\right) m^{n-1}+\left( \ell
_{a}^{\prime }+\ell _{b}^{\prime }\right)
\end{eqnarray*}%
so if $\ell _{a}^{\prime }+\ell _{b}^{\prime }<m^{n-1}$ then 
\begin{eqnarray*}
k_{n,m}(\ell _{a}+\ell _{b}) &=&k_{n,m}(\ell _{a})+k_{n,m}(\ell _{b}), \\
\left( \ell _{a}+\ell _{b}\right) ^{\prime } &=&\ell _{a}^{\prime }+\ell
_{b}^{\prime }\text{.}
\end{eqnarray*}%
And if $\ell _{a}^{\prime }+\ell _{b}^{\prime }\geq m^{n-1}$ 
\begin{eqnarray*}
k_{n,m}(\ell _{a}+\ell _{b}) &=&k_{n,m}(\ell _{a})+k_{n,m}(\ell _{b})+1, \\
\left( \ell _{a}+\ell _{b}\right) ^{\prime } &=&\ell _{a}^{\prime }+\ell
_{b}^{\prime }-m^{n-1}\text{.}
\end{eqnarray*}

Therefore 
\begin{equation*}
k_{n,m}(\ell _{a}+\ell _{b})=\left\{ 
\begin{tabular}{l}
$k_{n,m}(\ell _{a})+k_{n,m}(\ell _{b})$ \\ 
\ \ \ \ \ \ \ \ \ iff\ $\ \ell _{a}^{\prime }+\ell _{b}^{\prime }<m^{n-1}$
\\ 
$k_{n,m}(\ell _{a})+k_{n,m}(\ell _{b})+1$ \\ 
$\ \ \ \ \ \ \ \ \ $iff $\ \ell _{a}^{\prime }+\ell _{b}^{\prime }\geq
m^{n-1}$%
\end{tabular}%
\text{.}\right.
\end{equation*}
\end{proof}
\end{lemma}

Also,

\begin{corollary}
\begin{equation*}
\ (\ell _{a}+\ell _{b})^{\prime }=\left\{ 
\begin{tabular}{ll}
$\ell _{a}^{\prime }+\ell _{b}^{\prime }$ & iff $\ \ell _{a}^{\prime }+\ell
_{b}^{\prime }<m^{n-1}$ \\ 
$\ell _{a}^{\prime }+\ell _{b}^{\prime }-m^{n-1}$ & iff $\ \ell _{a}^{\prime
}+\ell _{b}^{\prime }\geq m^{n-1}$%
\end{tabular}%
.\right.
\end{equation*}
\end{corollary}

\begin{lemma}
If $\ell _{a}+\ell _{b}<m^{n}$, then $q_{n,m}(\ell _{a}+\ell _{b})=\left\{ 
\begin{tabular}{ll}
$q_{n,m}(\ell _{a})+q_{n,m}(\ell _{b})$ &  \\ 
\ \ \ \ \ \ \ \ \ \ \ \ or &  \\ 
$q_{n,m}(\ell _{a})+q_{n,m}(\ell _{b})-1$ & 
\end{tabular}%
\right. $.

\begin{proof}
\begin{eqnarray*}
q_{n,m}(\ell _{a}+\ell _{b}) &=&1+\left\lfloor \frac{\left( \left( \ell
_{a}+\ell _{b}\right) -1\right) \left( m-1\right) }{m^{n}-1}\right\rfloor 
\text{, by Lemma 2,} \\
&=&1+\left\lfloor \frac{\left( \ell _{a}-1\right) \left( m-1\right) }{m^{n}-1%
}+\frac{\ell _{b}\left( m-1\right) }{m^{n}-1}\right\rfloor \\
&\geq &1+\left\lfloor \frac{\left( \ell _{a}-1\right) \left( m-1\right) }{%
m^{n}-1}\right\rfloor +1+\left\lfloor \frac{\left( \left( \ell _{b}+1\right)
-1\right) \left( m-1\right) }{m^{n}-1}\right\rfloor -1 \\
&&\text{ \ \ \ \ \ \ \ (since }\left\lfloor x+y\right\rfloor \text{ is at
least }\left\lfloor x\right\rfloor +\left\lfloor y\right\rfloor \text{)} \\
&=&q_{n,m}(\ell _{a})+q_{n,m}(\ell _{b}+1)-1 \\
&\geq &q_{n,m}(\ell _{a})+q_{n,m}(\ell _{b})-1\text{ (by Remark 5).}
\end{eqnarray*}%
On the other hand,%
\begin{eqnarray*}
q_{n,m}(\ell _{a}+\ell _{b}) &=&1+\left\lfloor \frac{\left( \ell
_{a}-1\right) \left( m-1\right) }{m^{n}-1}+\frac{\ell _{b}\left( m-1\right) 
}{m^{n}-1}\right\rfloor \\
&&\text{ \ \ \ \ \ \ \ \ \ (since }\left\lfloor x+y\right\rfloor \text{ is
at most }\left\lfloor x\right\rfloor +\left\lfloor y\right\rfloor +1\text{)}
\\
&\leq &\left( 1+\left\lfloor \frac{\left( \ell _{a}-1\right) \left(
m-1\right) }{m^{n}-1}\right\rfloor \right) +\left( 1+\left\lfloor \frac{%
\left( \left( \ell _{b}+1\right) -1\right) \left( m-1\right) }{m^{n}-1}%
\right\rfloor \right) \\
&=&q_{n,m}(\ell _{a})+q_{n,m}(\ell _{b}+1) \\
&\leq &q_{n,m}(\ell )+q_{n,m}(\ell _{b})+1.
\end{eqnarray*}%
But $q_{n,m}(\ell _{a}+\ell _{b})=q_{n,m}(\ell _{a})+q_{n,m}(\ell _{b})+1$
only if $q_{n,m}(\ell _{b}+1)=q_{n,m}(\ell _{b})+1$ and then $\ell
_{b}=q_{n,m}(\ell _{b})\frac{m^{n-1}-1}{m-1}$. Also, since the condition on $%
\ell _{b}$ holds for $\ell _{a}$ ($\ell _{a},\ell _{b}$ are symmetric in $%
q_{n,m}(\ell _{a}+\ell _{b})=q_{n,m}(\ell _{a})+q_{n,m}(\ell _{b})+1$), $%
\ell _{a}=q_{n,m}(\ell _{a})\frac{m^{n-1}-1}{m-1}$. Therefore 
\begin{eqnarray*}
\ell _{a}+\ell _{b} &=&q_{n,m}(\ell _{a})\frac{m^{n-1}-1}{m-1}+q_{n,m}(\ell
_{b})\frac{m^{n-1}-1}{m-1} \\
&=&\left( q_{n,m}(\ell _{a})+q_{n,m}(\ell _{b})\right) \frac{m^{n-1}-1}{m-1}
\end{eqnarray*}%
and then 
\begin{eqnarray*}
q_{n,m}(\ell _{a}+\ell _{b}) &=&1+\left\lfloor \frac{\left( \left(
q_{n,m}(\ell _{a})+q_{n,m}(\ell _{b})\right) \frac{m^{n-1}-1}{m-1}-1\right)
\left( m-1\right) }{m^{n}-1}\right\rfloor \\
&=&1+\left\lfloor \left( q_{n,m}(\ell _{a})+q_{n,m}(\ell _{b})\right) -\frac{%
m-1}{m^{n-1}-1}\right\rfloor \\
&=&q_{n,m}(\ell _{a})+q_{n,m}(\ell _{b}).
\end{eqnarray*}%
This contradicts the supposition that $q_{n,m}(\ell _{a}+\ell
_{b})=q_{n,m}(\ell _{a})+q_{n,m}(\ell _{b})+1$.
\end{proof}
\end{lemma}

\begin{lemma}
If $\ell _{a}+\ell _{b}\geq m^{n}$, then $q_{n,m}(\ell _{a}+\ell _{b}-m^{n})$%
\ $=\left\{ 
\begin{tabular}{l}
$q_{n,m}(\ell _{a})+q_{n,m}(\ell _{b})-m$ \\ 
\ \ \ \ \ \ \ \ \ \ \ \ \ \ or \\ 
$q_{n,m}(\ell _{a})+q_{n,m}(\ell _{b})-m-1$%
\end{tabular}%
\right. .$

\begin{proof}
$q_{n,m}(\ell _{a}+\ell _{b}-m^{n})=m-q_{n,m}(m^{n}-\left( \ell _{a}+\ell
_{b}-m^{n}\right) )$ (by Remark 5)

\ \ \ \ \ \ \ \ \ \ \ \ \ \ \ \ \ \ \ \ \ \ \ \ \ \ \ \ \ $=m-q_{n,m}(\left(
m^{n}-\ell _{a}\right) +\left( m^{n}-\ell _{b}\right) )$

\ \ \ \ \ \ \ \ \ \ \ \ \ \ \ \ \ \ \ \ \ \ \ \ \ \ \ \ \ $=m-\left\{ 
\begin{tabular}{ll}
$q_{n,m}(m^{n}-\ell _{a})+q_{n,m}(m^{n}-\ell _{b})$ &  \\ 
\ \ \ \ \ \ \ \ \ \ \ \ or & (by Lemma 7) \\ 
$q_{n,m}(m^{n}-\ell _{a})+q_{n,m}(m^{n}-\ell _{b})-1$ & 
\end{tabular}%
\right. $

\ \ \ \ \ \ \ \ \ \ \ \ \ \ \ \ \ \ \ \ \ \ \ \ \ \ \ \ \ $=m-\left\{ 
\begin{tabular}{ll}
$\left( m-q_{n,m}(\ell _{a})\right) +\left( m-q_{n,m}(\ell _{b})\right) $ & 
\\ 
\ \ \ \ \ \ \ \ \ \ \ \ or & (by Remark 5) \\ 
$\left( m-q_{n,m}(\ell _{a})\right) +\left( m-q_{n,m}(\ell _{b})\right) -1$
& 
\end{tabular}%
\right. $

\ \ \ \ \ \ \ \ \ \ \ \ \ \ \ \ \ \ \ \ \ \ \ \ \ \ \ \ \ $=\left\{ 
\begin{tabular}{ll}
$q_{n,m}(\ell _{a})+q_{n,m}(\ell _{b})-m$ &  \\ 
\ \ \ \ \ \ \ \ \ \ \ \ or &  \\ 
$q_{n,m}(\ell _{a})+q_{n,m}(\ell _{b})-m-1$ & 
\end{tabular}%
\right. $.
\end{proof}
\end{lemma}

\subsection{Now the Proof (of Conjecture 3)}

Recall Conjecture 3: $\forall n,m,\ell _{a},\ell _{b}\in \mathbb{N}$,$%
\mathbb{\ }$such that $0\leq \ell _{b}\leq \ell _{a}\leq m^{n}$, if $\ell
_{a}+\ell _{b}\leq m^{n}$ then 
\begin{eqnarray*}
&&\left\vert \Theta (Lex^{-1}\left( n,m;\ell _{a}+\ell _{b}\right)
)\right\vert +\sigma _{n,m}\left( \ell _{a},\ell _{b}\right)  \\
&\leq &\left\vert \Theta (Lex^{-1}\left( n,m;\ell _{a}\right) )\right\vert
+\left\vert \Theta (Lex^{-1}\left( n,m;\ell _{b}\right) )\right\vert ,
\end{eqnarray*}

\ \ \ \ \ \ \ \ \ \ \ \ \ \ \ \ \ and if $\ell _{a}+\ell _{b}\geq m^{n}$
then 
\begin{eqnarray*}
&&\left\vert \Theta (Lex^{-1})\left( n,m;\ell _{a}+\ell _{b}-m^{n}\right)
\right\vert +\sigma _{n,m}\left( \ell _{a},\ell _{b}\right)  \\
&\leq &\left\vert \Theta (Lex^{-1})\left( n,m;\ell _{a}\right) \right\vert
+\left\vert \Theta (Lex^{-1})\left( n,m;\ell _{b}\right) \right\vert .
\end{eqnarray*}

\begin{proof}
The two inequalities are equivalent by duality: $\left\vert \Theta
(Lex^{-1}\left( n,m;m^{n}-\ell \right) )\right\vert =\left\vert \Theta
(Lex^{-1}\left( n,m;\ell \right) )\right\vert $ and (by Remark 5) $%
q_{n,m}\left( m^{n}-\ell \right) =m-q_{n,m}\left( \ell \right) $, so we need
only prove the first. Also, the inequality is trivial if $\ell _{b}=0$, so
we may assume that $\ell _{b}>0$. Letting 
\begin{eqnarray*}
\Sigma _{n,m}\left( \ell _{a},\ell _{b}\right) &=&\left\vert \Theta
(Lex^{-1}\left( n,m;\ell _{a}\right) )\right\vert +\left\vert \Theta
(Lex^{-1}\left( n,m;\ell _{b}\right) )\right\vert \\
&&-\left( \left\vert \Theta (Lex^{-1}\left( n,m;\ell _{a}+\ell _{b}\right)
)\right\vert +\sigma _{n,m}\left( \ell _{a},\ell _{b}\right) \right) ,
\end{eqnarray*}%
we must prove $\forall n,\Sigma _{n,m}\left( \ell _{a},\ell _{b}\right) \geq
0$. By induction on $n$:

\underline{Initial Case}: For $n=1$, $S(1,m)=K_{m}$ and every vertex is a
corner, so $q_{1,m}(\ell )=\ell $ and $\sigma _{1,m}\left( \ell _{a},\ell
_{b}\right) =\ell _{b}+\left( \left( \ell _{a}+\ell _{b}\right) -\ell
_{a}\right) =2\ell _{b}$. Also, $\left\vert \Theta (Lex^{-1}\left( 1,m;\ell
\right) )\right\vert =\ell \left( m-\ell \right) $ and $k_{1,m}\left( \ell
\right) =\ell $ so%
\begin{eqnarray*}
&&\Sigma _{1,m}\left( \ell _{a},\ell _{b}\right) =\ell _{a}\left( m-\ell
_{a}\right) +\ell _{b}\left( m-\ell _{b}\right) -\left( \left( \ell
_{a}+\ell _{b}\right) \left( m-\left( \ell _{a}+\ell _{b}\right) \right)
+2\ell _{b}\right) \\
&=&\allowbreak 2\ell _{b}\left( \ell _{a}-1\right) \\
&\geq &0\text{, since }\ell _{a}\geq \ell _{b}\geq 1\text{.}
\end{eqnarray*}

\underline{Inductive Step}: Assume $\Sigma _{n,m}\left( \ell _{a},\ell
_{b}\right) \geq 0$ for some $n\geq 1$. Then%
\begin{eqnarray*}
\Sigma _{n+1,m}\left( \ell _{a},\ell _{b}\right) &=&\left\vert \Theta
(Lex^{-1}\left( n+1,m;\ell _{a}\right) )\right\vert \\
&&+\left\vert \Theta (Lex^{-1}\left( n+1,m;\ell _{b}\right) )\right\vert \\
&&-\left\vert \Theta (Lex^{-1}\left( n+1,m;\ell _{a}+\ell _{b}\right)
)\right\vert \\
&&-\sigma _{n+1,m}\left( \ell _{a},\ell _{b}\right) \text{.}
\end{eqnarray*}%
By Lemma 3 (the recurrence for $\left\vert \Theta (Lex^{-1}\left\{
1,2,...,\ell \right\} )\right\vert $) it follows (since $\ell _{a}+\ell
_{b}\leq m^{n+1}$),

$\Sigma _{n+1,m}\left( \ell _{a},\ell _{b}\right) =$%
\begin{eqnarray*}
&&k_{n+1,m}(\ell _{a})\left( m-k_{n+1,m}(\ell _{a})\right) +\left\vert
\Theta (Lex^{-1}\left( n,m;\ell _{a}^{\prime }\right) )\right\vert \\
&&+\left\{ 
\begin{tabular}{ll}
$-q_{n,m}\left( \ell _{a}^{\prime }\right) $ & if $q_{n,m}\left( \ell
_{a}^{\prime }\right) \leq k_{n+1,m}(\ell _{a})$ \\ 
$q_{n,m}\left( \ell _{a}^{\prime }\right) -2k_{n+1,m}(\ell _{a})$ & if $%
q_{n,m}\left( \ell _{a}^{\prime }\right) >k_{n+1,m}(\ell _{a})$%
\end{tabular}%
\right. \\
+ &&k_{n+1,m}(\ell _{b})\left( m-k_{n+1,m}(\ell _{b})\right) +\left\vert
\Theta (Lex^{-1}\left( n,m;\ell _{b}^{\prime }\right) )\right\vert \\
&&+\left\{ 
\begin{tabular}{ll}
$-q_{n,m}\left( \ell _{b}^{\prime }\right) $ & if $q_{n,m}\left( \ell
_{b}^{\prime }\right) \leq k_{n+1,m}(\ell _{b})$ \\ 
$q_{n,m}\left( \ell _{b}^{\prime }\right) -2k_{n+1,m}(\ell _{b})$ & if $%
q_{n,m}\left( \ell _{b}^{\prime }\right) >k_{n+1,m}(\ell _{b})$%
\end{tabular}%
\right. \\
- &&k_{n+1,m}(\ell _{a}+\ell _{b})\left( m-k_{n+1,m}(\ell _{a}+\ell
_{b})\right) +\left\vert \Theta (Lex^{-1}\left( n,m;\left( \ell _{a}+\ell
_{b}\right) ^{\prime }\right) )\right\vert \\
&&-\left\{ 
\begin{tabular}{ll}
$-q_{n,m}\left( \left( \ell _{a}+\ell _{b}\right) ^{\prime }\right) $ & if $%
q_{n,m}\left( \left( \ell _{a}+\ell _{b}\right) ^{\prime }\right) \leq
k_{n+1,m}(\ell _{a}+\ell _{b})$ \\ 
$q_{n,m}\left( \left( \ell _{a}+\ell _{b}\right) ^{\prime }\right)
-2k_{n+1,m}(\ell _{a}+\ell _{b})$ & if $q_{n,m}\left( \left( \ell _{a}+\ell
_{b}\right) ^{\prime }\right) >k_{n+1,m}(\ell _{a}+\ell _{b})$%
\end{tabular}%
\right. \\
&&-\sigma _{n+1,m}\left( \ell _{a},\ell _{b}\right) \text{,}
\end{eqnarray*}%
where $0\leq \ell _{a}^{\prime },\ell _{b}^{\prime },\left( \ell _{a}+\ell
_{b}\right) ^{\prime }<m^{n}$. There is another binary conditional in this
formula: According to Lemma 7, $(\ell _{a}+\ell _{b})^{\prime }=\ell
_{a}^{\prime }+\ell _{b}^{\prime }$ iff $k_{n,m}(\ell _{a}+\ell
_{b})=k_{n,m}(\ell _{a})+k_{n,m}(\ell _{b})$ and this is true iff $\ell
_{a}^{\prime }+\ell _{b}^{\prime }<m^{n-1}.$ The other possibility is that $%
(\ell _{a}+\ell _{b})^{\prime }=\ell _{a}^{\prime }+\ell _{b}^{\prime
}-m^{n-1}$ which happens iff $\ell _{a}^{\prime }+\ell _{b}^{\prime }\geq
m^{n-1}$. Thus our formula for $\Sigma _{n+1,m}\left( \ell _{a},\ell
_{b}\right) $ has 4 binary conditionals, leading to $2^{4}=\allowbreak 16$
cases. We consider each of these 16 cases. To simplify the process we break
the formula into 4 pieces: $\Sigma _{n+1,m}\left( \ell _{a},\ell _{b}\right)
=$ $\mathbf{I+II+III+IV}$, where

\begin{description}
\item[ ] $\mathbf{I}$ consists of the terms derived from $k_{n+1,m}(\ell
)\left( m-k_{n+1,m}(\ell )\right) $ in the recurrence. If $\ell _{a}^{\prime
}+\ell _{b}^{\prime }<m^{n}$ then, by Lemma 6, $k_{n+1,m}(\ell _{a}+\ell
_{b})=k_{n+1,m}(\ell _{a})+k_{n+1,m}(\ell _{b})$ so

\ $\mathbf{I}$\ $=k_{n+1,m}(\ell _{a})\left( m-k_{n+1,m}(\ell _{a})\right)
+k_{n+1,m}(\ell _{b})\left( m-k_{n+1,m}(\ell _{b})\right) $

$\ \ \ \ \ -k_{n+1,m}(\ell _{a}+\ell _{b})\left( m-k_{n+1,m}(\ell _{a}+\ell
_{b})\right) $

\textbf{\ \ \ }$=k_{n+1,m}(\ell _{a})\left( m-k_{n+1,m}(\ell _{a})\right)
+k_{n+1,m}(\ell _{b})\left( m-k_{n+1,m}(\ell _{b})\right) $

$\ \ \ \ \ -\left( k_{n+1,m}(\ell _{a})+k_{n+1,m}(\ell _{b})\right) \left(
m-\left( k_{n+1,m}(\ell _{a})+k_{n+1,m}(\ell _{b})\right) \right) $

$\ =2k_{n+1,m}(\ell _{a})k_{n+1,m}(\ell _{b}).$

However, if $\ell _{a}^{\prime }+\ell _{b}^{\prime }\geq m^{n}$, then

$k_{n+1,m}(\ell _{a}+\ell _{b})=k_{n+1,m}(\ell _{a})+k_{n+1,m}(\ell _{b})+1$
so

$\mathbf{I=}$ $k_{n+1,m}(\ell _{a})\left( m-k_{n+1,m}(\ell _{a})\right)
+k_{n+1,m}(\ell _{b})\left( m-k_{n+1,m}(\ell _{b})\right) $

$\ \ \ \ \ -k_{n+1,m}(\ell _{a}+\ell _{b})\left( m-k_{n+1,m}(\ell _{a}+\ell
_{b})\right) $

\ \ $=k_{n+1,m}(\ell _{a})\left( m-k_{n+1,m}(\ell _{a})\right)
+k_{n+1,m}(\ell _{b})\left( m-k_{n+1,m}(\ell _{b})\right) $

$\ \ \ \ -\left( k_{n+1,m}(\ell _{a})+k_{n+1,m}(\ell _{b})+1\right) \left(
m-\left( k_{n+1,m}(\ell _{a})+k_{n+1,m}(\ell _{b})+1\right) \right) $

\ \ $=2\left( k_{n+1,m}(\ell _{a})k_{n+1,m}(\ell _{b})+k_{n+1,m}(\ell
_{a})+k_{n+1,m}(\ell _{b})\right) -m+1$.

$\mathbf{II}$ consists of the terms derived from $\left\vert \Theta
(Lex^{-1}\left( n,m;\ell ^{\prime }\right) )\right\vert $ in the recurrence
so $\ \ \ \ $

$\ \mathbf{II}=\left\vert \Theta (Lex^{-1}\left( n,m;\ell _{a}^{\prime
}\right) )\right\vert +\left\vert \Theta (Lex^{-1}\left( n,m;\ell
_{b}^{\prime }\right) )\right\vert $

$\ \ \ \ \ \ \ \ -\left\vert \Theta (Lex^{-1}\left( n,m;\left( \ell
_{a}+\ell _{b}\right) ^{\prime }\right) )\right\vert $

$\ \ =\left\vert \Theta (Lex^{-1}\left( n,m;\ell _{a}^{\prime }\right)
)\right\vert +\left\vert \Theta (Lex^{-1}\left( n,m;\ell _{b}^{\prime
}\right) )\right\vert $

$\ \ \ \ \ \ \ \ -\left\{ 
\begin{tabular}{ll}
$\left\vert \Theta (Lex^{-1}\left( n,m;\ell _{a}^{\prime }+\ell _{b}^{\prime
}\right) )\right\vert $ & if $\ell _{a}^{\prime }+\ell _{b}^{\prime }<m^{n}$
\\ 
$\left\vert \Theta (Lex^{-1}n,m;\ell _{a}^{\prime }+\ell _{b}^{\prime
}-m^{n})\right\vert $ & if $\ell _{a}^{\prime }+\ell _{b}^{\prime }\geq
m^{n} $%
\end{tabular}%
\right. $

$\ \ =\Sigma _{n,m}\left( \ell _{a}^{\prime },\ell _{b}^{\prime }\right)
+\sigma _{n,m}\left( \ell _{a}^{\prime },\ell _{b}^{\prime }\right) $ (by
the definition of $\Sigma _{n,m}$,

$\geq 0+\sigma _{n,m}\left( \ell _{a}^{\prime },\ell _{b}^{\prime }\right) $
(by the inductive hypothesis)

$=\left\{ 
\begin{tabular}{ll}
$q_{n,m}(\ell _{b}^{\prime })+\left( q_{n,m}(\ell _{a}^{\prime }+\ell
_{b}^{\prime })-q_{n,m}(\ell _{a}^{\prime })\right) $ & if $\ell
_{a}^{\prime }+\ell _{b}^{\prime }<m^{n}$ \\ 
$q_{n,m}(\ell _{b}^{\prime })-q_{n,m}(\ell _{a}^{\prime }+\ell _{b}^{\prime
}-m^{n})+\left( m-q_{n,m}(\ell _{a}^{\prime })\right) $ & if $\ell
_{a}^{\prime }+\ell _{b}^{\prime }\geq m^{n}$%
\end{tabular}%
\right. $

\ \ \ \ \ \ \ \ \ \ \ \ \ (by Definition 5).

$\mathbf{III}$ consists of the terms derived from

$\ \ \ \ \ \ \ \ \ \left\{ 
\begin{tabular}{ll}
$-q_{n-1,m}\left( \ell ^{\prime }\right) $ & if $q_{n-1,m}\left( \ell
^{\prime }\right) \leq k_{n,m}(\ell )$ \\ 
$q_{n-1,m}\left( \ell ^{\prime }\right) -2k_{n,m}(\ell )$ & if $%
q_{n-1,m}\left( \ell ^{\prime }\right) >k_{n,m}(\ell )$%
\end{tabular}%
\right. $

in the recurrence, so

$\mathbf{III}=\left\{ 
\begin{tabular}{ll}
$-q_{n,m}\left( \ell _{a}^{\prime }\right) $ & if $q_{n,m}\left( \ell
_{a}^{\prime }\right) \leq k_{n+1,m}(\ell _{a})$ \\ 
$q_{n,m}\left( \ell _{a}^{\prime }\right) -2k_{n+1,m}(\ell _{a})$ & if $%
q_{n,m}\left( \ell _{a}^{\prime }\right) >k_{n+1,m}(\ell _{a})$%
\end{tabular}%
\right. $

$\ \ \ \ \ +\left\{ 
\begin{tabular}{ll}
$-q_{n,m}\left( \ell _{b}^{\prime }\right) $ & if $q_{n,m}\left( \ell
_{b}^{\prime }\right) \leq k_{n+1,m}(\ell _{b})$ \\ 
$q_{n,m}\left( \ell _{b}^{\prime }\right) -2k_{n+1,m}(\ell _{b})$ & if $%
q_{n,m}\left( \ell _{b}^{\prime }\right) >k_{n+1,m}(\ell _{b})$%
\end{tabular}%
\right. $

$\ \ \ \ \ -\left\{ 
\begin{tabular}{l}
$-q_{n,m}\left( \left( \ell _{a}+\ell _{b}\right) ^{\prime }\right) $ \\ 
$\ \ \ \ \ \ \ \text{if }q_{n,m}\left( \left( \ell _{a}+\ell _{b}\right)
^{\prime }\right) \leq k_{n+1,m}(\ell _{a}+\ell _{b})$ \\ 
$q_{n,m}\left( \left( \ell _{a}+\ell _{b}\right) ^{\prime }\right)
-2k_{n+1,m}(\ell _{a}+\ell _{b})$ \\ 
\ \ \ \ \ \ \ if $q_{n,m}\left( \left( \ell _{a}+\ell _{b}\right) ^{\prime
}\right) >k_{n+1,m}(\ell _{a}+\ell _{b})$%
\end{tabular}%
\right. $,

$\mathbf{IV}=-\sigma _{n+1,m}\left( \ell _{a},\ell _{b}\right) $

$=-\left( q_{n+1,m}\left( \ell _{b}\right) +\left( q_{n+1,m}\left( \ell
_{a}+\ell _{b}\right) -q_{n+1,m}\left( \ell _{a}\right) \right) \right) $

$\geq -\left( q_{n+1,m}\left( \ell _{b}\right) +\left( q_{n+1,m}\left( \ell
_{a}\right) +q_{n+1,m}\left( \ell _{b}\right) \right) -q_{n+1,m}\left( \ell
_{a}\right) \right) $

\ \ \ \ \ \ \ \ \ \ \ (by Lemma 7)

$=-2q_{n+1,m}\left( \ell _{b}\right) $

\begin{remark}
1. $\mathbf{I}$ \& $\mathbf{II}$ are $\geq 0$, $\mathbf{III}$ may be $\leq 0$
or $>0$ and $\mathbf{IV\leq 0}$.

2. The magnitudes of terms in $\mathbf{I,II,III},\mathbf{IV}$ are mediated
by the Case inequalities, Remarks 1-5 and Lemmas 1-8. In each case we must
show that the negativity of $\mathbf{III},\mathbf{IV}$ is balanced out by
the positivity of $\mathbf{I,II}$.

3. There are 4 binary conditionals in the definition of $\Sigma
_{n+1,m}\left( \ell _{a},\ell _{b}\right) $:

\qquad\ \ i. $\ell _{a}^{\prime }+\ell _{b}^{\prime }<m^{n}$ or $\ell
_{a}^{\prime }+\ell _{b}^{\prime }\geq m^{n},$

\qquad\ ii. $q_{n,m}\left( \ell _{a}^{\prime }\right) \leq k_{n+1,m}(\ell
_{a})$ or $q_{n,m}\left( \ell _{a}^{\prime }\right) >k_{n+1,m}(\ell _{a}),$

\qquad iii. $q_{n,m}\left( \ell _{b}^{\prime }\right) \leq k_{n+1,m}(\ell
_{b})$ or $q_{n,m}\left( \ell _{b}^{\prime }\right) >k_{n+1,m}(\ell _{b}),$

\qquad iv. $q_{n,m}\left( \left( \ell _{a}+\ell _{b}\right) ^{\prime
}\right) \leq k_{n+1,m}(\left( \ell _{a}+\ell _{b}\right) ^{\prime })$ or $%
q_{n,m}\left( \left( \ell _{a}+\ell _{b}\right) ^{\prime }\right)
>k_{n+1,m}(\left( \ell _{a}+\ell _{b}\right) ^{\prime }).$

These give rise to 16 cases. In each case we must show that the negativity
of $\mathbf{III},\mathbf{IV}$ is balanced out by the positivity of $\mathbf{%
I,II}$. We associate these 16 cases with the binary 4-tuples of 1s (first
case) and 2s (second case) and consider them in lexicographic order:

Case 1111: $\ell _{a}^{\prime }+\ell _{b}^{\prime }<m^{n}$, $q_{n,m}\left(
\ell _{a}^{\prime }\right) \leq k_{n+1,m}(\ell _{a})$, $q_{n,m}\left( \ell
_{b}^{\prime }\right) \leq k_{n+1,m}(\ell _{b})$, $q_{n,m}\left( \left( \ell
_{a}+\ell _{b}\right) ^{\prime }\right) \leq k_{n+1,m}(\left( \ell _{a}+\ell
_{b}\right) )$.
\end{remark}
\end{description}

\qquad Then 
\begin{eqnarray*}
\Sigma _{n+1,m}\left( \ell _{a},\ell _{b}\right) &=&\text{I + II + III + IV}
\\
&\geq &2k_{n+1,m}(\ell _{a})k_{n+1,m}(\ell _{b}) \\
&&+q_{n,m}\left( \ell _{b}^{\prime }\right) +q_{n,m}\left( \ell _{a}^{\prime
}+\ell _{b}^{\prime }\right) -q_{n,m}\left( \ell _{a}^{\prime }\right) \text{
} \\
&&\text{ }-q_{n,m}\left( \ell _{a}^{\prime }\right) -q_{n,m}\left( \ell
_{b}^{\prime }\right) +q_{n,m}\left( \ell _{a}^{\prime }+\ell _{b}^{\prime
}\right) \\
&&-2q_{n+1,m}\left( \ell _{b}\right) , \\
&\geq &2k_{n+1,m}(\ell _{a})k_{n+1,m}(\ell _{b}) \\
&&+2q_{n,m}\left( \ell _{a}^{\prime }+\ell _{b}^{\prime }\right)
-2q_{n,m}\left( \ell _{a}^{\prime }\right) \\
&&-2\left( k_{n+1,m}\left( \ell _{b}\right) +1\right) \text{ (by Remark 3)}
\\
&\geq &2k_{n+1,m}(\ell _{a})k_{n+1,m}(\ell _{b})-2\left( k_{n+1,m}\left(
\ell _{b}\right) +1\right) \\
&&\text{ \ \ \ \ \ \ \ \ \ (by Remark 5, }q_{n,m}\left( \ell _{a}^{\prime
}+\ell _{b}^{\prime }\right) -q_{n,m}\left( \ell _{a}^{\prime }\right) \geq 0%
\text{)} \\
&=&2\left( k_{n+1,m}(\ell _{a})k_{n+1,m}(\ell _{b})-k_{n+1,m}\left( \ell
_{b}\right) -1\right) \\
&\geq &0\text{ unless }k_{n+1,m}\left( \ell _{b}\right) =0\text{ or }%
k_{n+1,m}(\ell _{a})=1=k_{n+1,m}(\ell _{b}).
\end{eqnarray*}%
But if $k_{n+1,m}\left( \ell _{b}\right) =0$ then by Case 1111.iii, $%
q_{n,m}\left( \ell _{b}^{\prime }\right) \leq k_{n+1,m}\left( \ell
_{b}\right) =0$ so $q_{n,m}\left( \ell _{b}^{\prime }\right) =0$ which
implies that $\ell _{b}^{\prime }=0.$ Therefore $\ell _{b}=k_{n+1,m}\left(
\ell _{b}\right) m^{n}+\ell _{b}^{\prime }=0$, contradicting our assumption
that $\ell _{b}>0$. And if $k_{n+1,m}(\ell _{a})=1=k_{n+1,m}(\ell _{b})$
then by the inequalities above, 
\begin{eqnarray*}
\Sigma _{n+1,m}\left( \ell _{a},\ell _{b}\right) &\geq &2\cdot 1\cdot 1 \\
&&+2q_{n,m}\left( \ell _{a}^{\prime }+\ell _{b}^{\prime }\right)
-2q_{n,m}\left( \ell _{a}^{\prime }\right) \\
&&-2q_{n+1,m}\left( \ell _{b}\right) ,\text{ }
\end{eqnarray*}

Since $k_{n+1,m}(\ell _{b})=1,$ $m^{n}\leq \ell _{b}<2m^{n}$. If $\ell
_{b}\leq \frac{m^{n+1}-1}{m-1}$, then $q_{n+1,m}\left( \ell _{b}\right) =1.$
And if $\ell _{b}^{\prime }=0$ then $q_{n,m}\left( \ell _{a}^{\prime }+\ell
_{b}^{\prime }\right) =q_{n,m}\left( \ell _{a}^{\prime }\right) $ so 
\begin{equation*}
\Sigma _{n+1,m}\left( \ell _{a},\ell _{b}\right) \geq 2+2\cdot 0-2\cdot 1=0%
\text{.}
\end{equation*}%
However, if $\ell _{b}^{\prime }>0$, then $q_{n,m}\left( \ell _{b}^{\prime
}\right) \geq 1$ so%
\begin{eqnarray*}
\Sigma _{n+1,m}\left( \ell _{a},\ell _{b}\right) &\geq &2 \\
&&+2\left( q_{n,m}\left( \ell _{a}^{\prime }+\ell _{b}^{\prime }\right)
-q_{n,m}\left( \ell _{a}^{\prime }\right) \right) \\
&&-2q_{n+1,m}\left( \ell _{b}\right) ,\text{ } \\
&\geq &2 \\
&&+2\left( \left( q_{n,m}\left( \ell _{a}^{\prime }\right) +q_{n,m}\left(
\ell _{b}^{\prime }\right) -1\right) -q_{n,m}\left( \ell _{a}^{\prime
}\right) \right) \text{ (by Lemma 7)} \\
&&-2\cdot 1, \\
&=&2\left( q_{n,m}\left( \ell _{b}^{\prime }\right) -1\right) \\
&\geq &0
\end{eqnarray*}

And if $\ell _{b}>\frac{m^{n+1}-1}{m-1}=m^{n}+\frac{m^{n}-1}{m-1}$ then $%
\ell _{b}^{\prime }>\frac{m^{n}-1}{m-1}$ so $q_{n+1,m}\left( \ell
_{b}\right) =2$ and $q_{n,m}\left( \ell _{b}^{\prime }\right) \geq 2$.
Therefore%
\begin{eqnarray*}
\Sigma _{n+1,m}\left( \ell _{a},\ell _{b}\right) &\geq &2 \\
&&+2\left( q_{n,m}\left( \ell _{a}^{\prime }+\ell _{b}^{\prime }\right)
-q_{n,m}\left( \ell _{a}^{\prime }\right) \right) \\
&&-2q_{n+1,m}\left( \ell _{b}\right) ,\text{ } \\
&\geq &2 \\
&&+2\left( \left( q_{n,m}\left( \ell _{a}^{\prime }\right) +q_{n,m}\left(
\ell _{b}^{\prime }\right) -1\right) -q_{n,m}\left( \ell _{a}^{\prime
}\right) \right) \text{ (by Lemma 7)} \\
&&-2\cdot 2, \\
&=&2\left( q_{n,m}\left( \ell _{b}^{\prime }\right) -2\right) \\
&\geq &0
\end{eqnarray*}%
So in any subcase of Case 1111, $\Sigma _{n+1,m}\left( \ell _{a},\ell
_{b}\right) \geq 0$.

\medskip \bigskip

Case 1112: $\ell _{a}^{\prime }+\ell _{b}^{\prime }<m^{n}$, $q_{n,m}\left(
\ell _{a}^{\prime }\right) \leq k_{n+1,m}(\ell _{a})$, $q_{n,m}\left( \ell
_{b}^{\prime }\right) \leq k_{n+1,m}(\ell _{b})$, $q_{n,m}\left( \left( \ell
_{a}+\ell _{b}\right) ^{\prime }\right) >k_{n+1,m}(\left( \ell _{a}+\ell
_{b}\right) )$.

\qquad Then 
\begin{eqnarray*}
\Sigma _{n+1,m}\left( \ell _{a},\ell _{b}\right) &=&\text{I + II + III + IV}
\\
&\geq &2k_{n+1,m}(\ell _{a})k_{n+1,m}(\ell _{b}) \\
&&+q_{n,m}\left( \ell _{b}^{\prime }\right) +q_{n,m}\left( \ell _{a}^{\prime
}+\ell _{b}^{\prime }\right) -q_{n,m}\left( \ell _{a}^{\prime }\right) \\
&&-q_{n,m}\left( \ell _{a}^{\prime }\right) -q_{n,m}\left( \ell _{b}^{\prime
}\right) -(q_{n,m}\left( \ell _{a}^{\prime }+\ell _{b}^{\prime }\right)
-2k_{n+1,m}(\ell _{a}+\ell _{b})) \\
&&-2q_{n+1,m}\left( \ell _{b}\right) \text{,} \\
&\geq &2k_{n+1,m}(\ell _{a})k_{n+1,m}(\ell _{b}) \\
&&-2q_{n,m}\left( \ell _{a}^{\prime }\right) +2k_{n+1,m}(\ell _{a}+\ell _{b})
\\
&&-2\left( k_{n+1,m}\left( \ell _{b}\right) +1\right) \text{ \ \ \ (by
Remark 3)} \\
&\geq &2k_{n+1,m}(\ell _{a})k_{n+1,m}(\ell _{b}) \\
&&-2k_{n+1,m}(\ell _{a})+2\left( k_{n+1,m}(\ell _{a})+k_{n+1,m}(\ell
_{b})\right) \\
&&\text{ \ \ \ \ \ \ \ \ \ \ \ \ \ \ \ (by Case 1112.ii \& Lemma 6)} \\
&&-2\left( k_{n+1,m}\left( \ell _{b}\right) +1\right) , \\
&=&2\left( k_{n+1,m}(\ell _{a})k_{n+1,m}(\ell _{b})-1\right) \\
&\geq &0\text{ unless }k_{n+1,m}(\ell _{b})=0\text{ but that is impossible
for the} \\
&&\text{ \ \ \ \ \ \ \ \ \ \ \ \ same reason as in Case 1111.}
\end{eqnarray*}

\newpage

Case 1121: $\ell _{a}^{\prime }+\ell _{b}^{\prime }<m^{n}$, $q_{n,m}\left(
\ell _{a}^{\prime }\right) \leq k_{n+1,m}(\ell _{a})$, $q_{n,m}\left( \ell
_{b}^{\prime }\right) >k_{n+1,m}(\ell _{b})$, $q_{n,m}\left( \left( \ell
_{a}+\ell _{b}\right) ^{\prime }\right) \leq k_{n+1,m}(\ell _{a}+\ell _{b})$.

\qquad Then%
\begin{eqnarray*}
\Sigma _{n+1,m}\left( \ell _{a},\ell _{b}\right) &=&\text{I + II + III + IV}
\\
&\geq &2k_{n+1,m}(\ell _{a})k_{n+1,m}(\ell _{b}) \\
&&+q_{n,m}\left( \ell _{b}^{\prime }\right) +q_{n,m}\left( \ell _{a}^{\prime
}+\ell _{b}^{\prime }\right) -q_{n,m}\left( \ell _{a}^{\prime }\right) \\
&&+\left( -q_{n,m}\left( \ell _{a}^{\prime }\right) \right) +\left(
q_{n,m}\left( \ell _{b}^{\prime }\right) -2k_{n+1,m}(\ell _{b}\right)
-(-q_{n,m}\left( \ell _{a}^{\prime }+\ell _{b}^{\prime }\right) \\
&&-2q_{n+1,m}\left( \ell _{b}\right) , \\
&=&2k_{n+1,m}(\ell _{a})k_{n+1,m}(\ell _{b}) \\
&&+2q_{n,m}\left( \ell _{b}^{\prime }\right) -2q_{n,m}\left( \ell
_{a}^{\prime }\right) +2q_{n,m}\left( \ell _{a}^{\prime }+\ell _{b}^{\prime
}\right) \\
&&-2k_{n+1,m}\left( \ell _{b}\right) \\
&&-2\left( k_{n+1,m}\left( \ell _{b}\right) +1\right) \text{ (by Remark 3),}
\\
&\geq &2k_{n+1,m}(\ell _{a})k_{n+1,m}(\ell _{b}) \\
&&+2q_{n,m}\left( \ell _{b}^{\prime }\right) -2q_{n,m}\left( \ell
_{a}^{\prime }\right) +2\left( q_{n,m}\left( \ell _{a}^{\prime }\right)
+q_{n,m}\left( \ell _{b}^{\prime }\right) -1\right) \\
&&\text{ \ \ \ \ \ \ \ \ \ \ \ \ \ \ \ \ \ (by Lemma 7)} \\
&&-2\left( 2k_{n+1,m}\left( \ell _{b}\right) +1\right) \\
&=&2k_{n+1,m}(\ell _{a})k_{n+1,m}(\ell _{b}) \\
&&+4q_{n,m}\left( \ell _{b}^{\prime }\right) \\
&&-4\left( k_{n+1,m}\left( \ell _{b}\right) +1\right) \\
&\geq &2k_{n+1,m}(\ell _{a})k_{n+1,m}(\ell _{b})\text{ (by Case 1121.iii)} \\
&\geq &0.
\end{eqnarray*}

\bigskip

Case 1122: $\ell _{a}^{\prime }+\ell _{b}^{\prime }<m^{n}$, $q_{n,m}\left(
\ell _{a}^{\prime }\right) \leq k_{n+1,m}(\ell _{a})$, $q_{n,m}\left( \ell
_{b}^{\prime }\right) >k_{n+1,m}(\ell _{b})$, $q_{n,m}\left( \left( \ell
_{a}+\ell _{b}\right) ^{\prime }\right) >k_{n+1,m}(\left( \ell _{a}+\ell
_{b}\right) )$.

\qquad Then%
\begin{eqnarray*}
\Sigma _{n+1,m}\left( \ell _{a},\ell _{b}\right) &=&\text{I + II + III + IV}
\\
&\geq &2k_{n+1,m}(\ell _{a})k_{n+1,m}(\ell _{b}) \\
&&+q_{n,m}\left( \ell _{b}^{\prime }\right) +q_{n,m}\left( \ell _{a}^{\prime
}+\ell _{b}^{\prime }\right) -q_{n,m}\left( \ell _{a}^{\prime }\right) \\
&&+\left( -q_{n,m}\left( \ell _{a}^{\prime }\right) \right) +\left(
q_{n,m}\left( \ell _{b}^{\prime }\right) -2k_{n+1,m}(\ell _{b})\right) \\
&&-(q_{n,m}\left( \ell _{a}^{\prime }+\ell _{b}^{\prime }\right)
-2k_{n+1,m}(\ell _{a}+\ell _{b})) \\
&&-2q_{n+1,m}\left( \ell _{b}\right) , \\
&\geq &2k_{n+1,m}(\ell _{a})k_{n+1,m}(\ell _{b}) \\
&&+2q_{n,m}\left( \ell _{b}^{\prime }\right) -2q_{n,m}\left( \ell
_{a}^{\prime }\right) \\
&&-2k_{n+1,m}(\ell _{b})+2k_{n+1,m}(\ell _{a}+\ell _{b}) \\
&&-2\left( k_{n+1,m}\left( \ell _{b}\right) +1\right) \text{ (by Remark 3),}
\\
&\geq &2k_{n+1,m}(\ell _{a})k_{n+1,m}(\ell _{b})-2k_{n+1,m}\left( \ell
_{b}\right) -2 \\
&&+2\left( k_{n+1,m}\left( \ell _{b}\right) +1\right) -2k_{n+1,m}(\ell _{a})%
\text{ (by Cases 1122.ii \& .iii)} \\
&&-2k_{n+1,m}(\ell _{b})+2\left( k_{n+1,m}(\ell _{a})+k_{n+1,m}(\ell
_{b})\right) \text{ (by Lemma 6)} \\
&=&2k_{n+1,m}(\ell _{a})k_{n+1,m}(\ell _{b}) \\
&\geq &0\text{.}
\end{eqnarray*}

Case 1211: $\ell _{a}^{\prime }+\ell _{b}^{\prime }<m^{n}$, $q_{n,m}\left(
\ell _{a}^{\prime }\right) >k_{n+1,m}(\ell _{a})$, $q_{n,m}\left( \ell
_{b}^{\prime }\right) \leq k_{n+1,m}(\ell _{b})$, $q_{n,m}\left( \left( \ell
_{a}+\ell _{b}\right) ^{\prime }\right) \leq k_{n+1,m}(\left( \ell _{a}+\ell
_{b}\right) )$.

\qquad Then%
\begin{eqnarray*}
\Sigma _{n+1,m}\left( \ell _{a},\ell _{b}\right) &=&\text{I + II + III + IV}
\\
&\geq &2k_{n+1,m}(\ell _{a})k_{n+1,m}(\ell _{b}) \\
&&+q_{n,m}\left( \ell _{b}^{\prime }\right) +q_{n,m}\left( \ell _{a}^{\prime
}+\ell _{b}^{\prime }\right) -q_{n,m}\left( \ell _{a}^{\prime }\right) \\
&&+\left( q_{n,m}\left( \ell _{a}^{\prime }\right) -2k_{n+1,m}(\ell
_{a})\right) +\left( -q_{n,m}\left( \ell _{b}^{\prime }\right) \right)
-(-q_{n,m}\left( \ell _{a}^{\prime }+\ell _{b}^{\prime }\right) ) \\
&&-2q_{n+1,m}\left( \ell _{b}\right) , \\
&\geq &2k_{n+1,m}(\ell _{a})k_{n+1,m}(\ell _{b}) \\
&&+2q_{n,m}\left( \ell _{a}^{\prime }+\ell _{b}^{\prime }\right)
-2k_{n+1,m}(\ell _{a}) \\
&&-2\left( k_{n+1,m}\left( \ell _{b}\right) +1\right) \text{ (by Remark 3)}
\\
&\geq &2k_{n+1,m}(\ell _{a})k_{n+1,m}(\ell _{b}) \\
&&2q_{n,m}\left( \ell _{a}^{\prime }\right) -2k_{n+1,m}(\ell _{a})\text{ (by
Remark 3)} \\
&&-2\left( k_{n+1,m}\left( \ell _{b}\right) +1\right) , \\
&\geq &2k_{n+1,m}(\ell _{a})k_{n+1,m}(\ell _{b}) \\
&&2\left( k_{n+1,m}\left( \ell _{a}\right) +1\right) -2k_{n+1,m}(\ell _{a})%
\text{ (by Case 1211.ii)} \\
&&-2\left( k_{n+1,m}\left( \ell _{b}\right) +1\right) \text{,} \\
&=&2\left( k_{n+1,m}(\ell _{a})k_{n+1,m}(\ell _{b})-k_{n+1,m}\left( \ell
_{b}\right) \right) \text{,} \\
&\geq &0\text{ since }k_{n+1,m}\left( \ell _{a}\right) \geq k_{n+1,m}(\ell
_{b})\geq 0\text{.}
\end{eqnarray*}

\bigskip

Case 1212: $\ell _{a}^{\prime }+\ell _{b}^{\prime }<m^{n}$, $q_{n,m}\left(
\ell _{a}^{\prime }\right) >k_{n+1,m}(\ell _{a})$, $q_{n,m}\left( \ell
_{b}^{\prime }\right) \leq k_{n+1,m}(\ell _{b})$, $q_{n,m}\left( \left( \ell
_{a}+\ell _{b}\right) ^{\prime }\right) >k_{n+1,m}(\left( \ell _{a}+\ell
_{b}\right) )$.

\qquad Then%
\begin{eqnarray*}
\Sigma _{n+1,m}\left( \ell _{a},\ell _{b}\right) &=&\text{I + II + III + IV}
\\
&\geq &2k_{n+1,m}(\ell _{a})k_{n+1,m}(\ell _{b}) \\
&&+q_{n,m}\left( \ell _{b}^{\prime }\right) +q_{n,m}\left( \ell _{a}^{\prime
}+\ell _{b}^{\prime }\right) -q_{n,m}\left( \ell _{a}^{\prime }\right) \\
&&+\left( q_{n,m}\left( \ell _{a}^{\prime }\right) -2k_{n+1,m}(\ell
_{a})\right) +\left( -q_{n,m}\left( \ell _{b}^{\prime }\right) \right) \\
&&-(q_{n,m}\left( \ell _{a}^{\prime }+\ell _{b}^{\prime }\right)
-2k_{n+1,m}(\ell _{a}+\ell _{b})) \\
&&-2q_{n+1,m}\left( \ell _{b}\right) , \\
&=&2k_{n+1,m}(\ell _{a})k_{n+1,m}(\ell _{b}) \\
&&-2k_{n+1,m}(\ell _{a})+2k_{n+1,m}(\ell _{a}+\ell _{b}) \\
&&-2\left( k_{n+1,m}\left( \ell _{b}\right) +1\right) \text{ (by Remark 3)}
\\
&\geq &2k_{n+1,m}(\ell _{a})k_{n+1,m}(\ell _{b}) \\
&&-2k_{n+1,m}(\ell _{a})+2\left( k_{n+1,m}(\ell _{a})+k_{n+1,m}(\ell
_{b})\right) \text{ (by Lemma 6)} \\
&&-2\left( k_{n+1,m}\left( \ell _{b}\right) +1\right) , \\
&\geq &2\left( k_{n+1,m}(\ell _{a})k_{n+1,m}(\ell _{b})-1\right) \\
&\geq &0\text{ unless }k_{n+1,m}(\ell _{b})=0\text{.}
\end{eqnarray*}

But $k_{n+1,m}\left( \ell _{b}\right) =0$ leads to a contradiction as it did
in Case 1111.

\newpage

Case 1221: $\ell _{a}^{\prime }+\ell _{b}^{\prime }<m^{n}$, $q_{n,m}\left(
\ell _{a}^{\prime }\right) >k_{n+1,m}(\ell _{a})$, $q_{n,m}\left( \ell
_{b}^{\prime }\right) >k_{n+1,m}(\ell _{b})$, $q_{n,m}\left( \left( \ell
_{a}+\ell _{b}\right) ^{\prime }\right) \leq k_{n+1,m}(\left( \ell _{a}+\ell
_{b}\right) )$.

\qquad Then%
\begin{eqnarray*}
\Sigma _{n+1,m}\left( \ell _{a},\ell _{b}\right) &=&\text{I + II + III + IV}
\\
&\geq &2k_{n+1,m}(\ell _{a})k_{n+1,m}(\ell _{b}) \\
&&+q_{n,m}\left( \ell _{b}^{\prime }\right) +q_{n,m}\left( \ell _{a}^{\prime
}+\ell _{b}^{\prime }\right) -q_{n,m}\left( \ell _{a}^{\prime }\right) \\
&&+\left( q_{n,m}\left( \ell _{a}^{\prime }\right) -2k_{n+1,m}(\ell
_{a})\right) +\left( q_{n,m}\left( \ell _{b}^{\prime }\right)
-2k_{n+1,m}(\ell _{b})\right) \\
&&-(-q_{n,m}\left( \ell _{a}^{\prime }+\ell _{b}^{\prime }\right) ) \\
&&-2q_{n+1,m}\left( \ell _{b}\right) , \\
&=&2k_{n+1,m}(\ell _{a})k_{n+1,m}(\ell _{b}) \\
&&+2q_{n,m}\left( \ell _{b}^{\prime }\right) +2q_{n,m}\left( \ell
_{a}^{\prime }+\ell _{b}^{\prime }\right) \\
&&-2k_{n+1,m}(\ell _{a})-2k_{n+1,m}(\ell _{b}) \\
&&-2\left( k_{n+1,m}\left( \ell _{b}\right) +1\right) \text{ (by Remark 3)}
\\
&\geq &2k_{n+1,m}(\ell _{a})k_{n+1,m}(\ell _{b}) \\
&&+2q_{n,m}\left( \ell _{b}^{\prime }\right) +2\left( q_{n,m}\left( \ell
_{a}^{\prime }\right) +q_{n,m}\left( \ell _{b}^{\prime }\right) -1\right) 
\text{ (by Lemma 7)} \\
&&-2k_{n+1,m}(\ell _{a})-2k_{n+1,m}(\ell _{b}) \\
&&-2\left( k_{n+1,m}\left( \ell _{b}\right) +1\right) , \\
&\geq &2k_{n+1,m}(\ell _{a})k_{n+1,m}(\ell _{b}) \\
&&+2\left( k_{n+1,m}\left( \ell _{b}\right) +1\right) +2\left( \left(
k_{n+1,m}\left( \ell _{a}\right) +1\right) +k_{n+1,m}\left( \ell _{b}\right)
+1-1\right) \\
&&\text{ \ \ \ \ \ \ \ \ \ \ \ \ \ \ \ \ \ \ \ \ (by Cases 1221.ii \& .iii)}
\\
&&-2k_{n+1,m}(\ell _{a})-2k_{n+1,m}(\ell _{b}) \\
&&-2\left( k_{n+1,m}(\ell _{b})+1\right) , \\
&=&2\left( k_{n+1,m}(\ell _{a})k_{n+1,m}(\ell _{b})+1\right) \\
&>&0
\end{eqnarray*}

\newpage

Case 1222: $\ell _{a}^{\prime }+\ell _{b}^{\prime }<m^{n}$, $q_{n,m}\left(
\ell _{a}^{\prime }\right) >k_{n+1,m}(\ell _{a})$, $q_{n,m}\left( \ell
_{b}^{\prime }\right) >k_{n+1,m}(\ell _{b})$, $q_{n,m}\left( \left( \ell
_{a}+\ell _{b}\right) ^{\prime }\right) >k_{n+1,m}(\left( \ell _{a}+\ell
_{b}\right) )$.

\qquad Then%
\begin{eqnarray*}
\Sigma _{n+1,m}\left( \ell _{a},\ell _{b}\right) &=&\text{I + II + III + IV}
\\
&\geq &2k_{n+1,m}(\ell _{a})k_{n+1,m}(\ell _{b}) \\
&&+q_{n,m}\left( \ell _{b}^{\prime }\right) +q_{n,m}\left( \ell _{a}^{\prime
}+\ell _{b}^{\prime }\right) -q_{n,m}\left( \ell _{a}^{\prime }\right) \\
&&+\left( q_{n,m}\left( \ell _{a}^{\prime }\right) -2k_{n+1,m}(\ell
_{a})\right) +\left( q_{n,m}\left( \ell _{b}^{\prime }\right)
-2k_{n+1,m}(\ell _{b})\right) \\
&&-(q_{n,m}\left( \ell _{a}^{\prime }+\ell _{b}^{\prime }\right)
-2k_{n+1,m}(\ell _{a}+\ell _{b})) \\
&&-2q_{n+1,m}\left( \ell _{b}\right) , \\
&\geq &2k_{n+1,m}(\ell _{a})k_{n+1,m}(\ell _{b}) \\
&&+2q_{n,m}\left( \ell _{b}^{\prime }\right) \\
&&-2k_{n+1,m}(\ell _{a})-2k_{n+1,m}(\ell _{b}) \\
&&+2k_{n+1,m}(\ell _{a}+\ell _{b})) \\
&&-2\left( k_{n+1,m}\left( \ell _{b}\right) +1\right) ,\text{ (by Remark 3)}
\\
&\geq &2k_{n+1,m}(\ell _{a})k_{n+1,m}(\ell _{b}) \\
&&+2\left( k_{n+1,m}(\ell _{b})+1\right) \text{ (by Case 1222.iii)} \\
&&-2k_{n+1,m}(\ell _{a})-2k_{n+1,m}(\ell _{b}) \\
&&+2\left( k_{n+1,m}(\ell _{a})+k_{n+1,m}(\ell _{b})\right) \text{ (by Lemma
6)} \\
&&-2\left( k_{n+1,m}\left( \ell _{b}\right) +1\right) , \\
&=&2k_{n+1,m}(\ell _{a})k_{n+1,m}(\ell _{b}) \\
&\geq &0\text{.}
\end{eqnarray*}

\bigskip

Case 2111: $\ell _{a}^{\prime }+\ell _{b}^{\prime }\geq m^{n}$, $%
q_{n,m}\left( \ell _{a}^{\prime }\right) \leq k_{n+1,m}(\ell _{a})$, $%
q_{n,m}\left( \ell _{b}^{\prime }\right) \leq k_{n+1,m}(\ell _{b})$, $%
q_{n,m}\left( \left( \ell _{a}+\ell _{b}\right) ^{\prime }\right) \leq
k_{n+1,m}(\ell _{a}+\ell _{b})$.

\qquad Then%
\begin{eqnarray*}
\Sigma _{n+1,m}\left( \ell _{a},\ell _{b}\right) &=&\text{I + II + III + IV}
\\
&\geq &2\left( k_{n+1,m}(\ell _{a})k_{n+1,m}(\ell _{b})+k_{n+1,m}(\ell
_{a})+k_{n+1,m}(\ell _{b})\right) -m+1 \\
&&+q_{n,m}\left( \ell _{b}^{\prime }\right) -q_{n,m}\left( \ell _{a}^{\prime
}+\ell _{b}^{\prime }-m^{n}\right) +m-q_{n,m}\left( \ell _{a}^{\prime
}\right) \text{ } \\
&&+\left( -q_{n,m}\left( \ell _{a}^{\prime }\right) \right) +\left(
-q_{n,m}\left( \ell _{b}^{\prime }\right) \right) -\left( -q_{n,m}\left(
\ell _{a}^{\prime }+\ell _{b}^{\prime }-m^{n}\right) \right) \\
&&-2q_{n+1,m}\left( \ell _{b}\right) , \\
&\geq &2\left( k_{n+1,m}(\ell _{a})k_{n+1,m}(\ell _{b})+k_{n+1,m}(\ell
_{a})+k_{n+1,m}(\ell _{b})\right) +1 \\
&&-2q_{n,m}\left( \ell _{a}^{\prime }\right) \\
&&-2\left( k_{n+1,m}\left( \ell _{b}\right) +1\right) ,\text{ (by Remark 3)}
\\
&\geq &2\left( k_{n+1,m}(\ell _{a})k_{n+1,m}(\ell _{b})+k_{n+1,m}(\ell
_{a})+k_{n+1,m}(\ell _{b})\right) +1 \\
&&-2k_{n+1,m}\left( \ell _{a}\right) \text{ (by Case 2111.ii)} \\
&&-2\left( k_{n+1,m}\left( \ell _{b}\right) +1\right) , \\
&=&2k_{n+1,m}(\ell _{a})k_{n+1,m}(\ell _{b})-1 \\
&\geq &0\text{ unless }k_{n+1,m}\left( \ell _{b}\right) =0\text{.}
\end{eqnarray*}%
But $k_{n+1,m}\left( \ell _{b}\right) =0$ is impossible for the same reason
it was in Case 1111.

\bigskip \bigskip

Case 2112: $\ell _{a}^{\prime }+\ell _{b}^{\prime }\geq m^{n}$, $%
q_{n,m}\left( \ell _{a}^{\prime }\right) \leq k_{n+1,m}(\ell _{a})$, $%
q_{n,m}\left( \ell _{b}^{\prime }\right) \leq k_{n+1,m}(\ell _{b})$, $%
q_{n,m}\left( \left( \ell _{a}+\ell _{b}\right) ^{\prime }\right)
>k_{n+1,m}(\ell _{a}+\ell _{b})$.

\qquad Then%
\begin{eqnarray*}
\Sigma _{n+1,m}\left( \ell _{a},\ell _{b}\right) &=&\text{I + II + III + IV}
\\
&\geq &2\left( k_{n+1,m}(\ell _{a})k_{n+1,m}(\ell _{b})+k_{n+1,m}(\ell
_{a})+k_{n+1,m}(\ell _{b})\right) -m+1 \\
&&+q_{n,m}\left( \ell _{b}^{\prime }\right) -q_{n,m}\left( \ell _{a}^{\prime
}+\ell _{b}^{\prime }-m^{n}\right) +m-q_{n,m}\left( \ell _{a}^{\prime
}\right) \\
&&+\left( -q_{n,m}\left( \ell _{a}^{\prime }\right) \right) +\left(
-q_{n,m}\left( \ell _{b}^{\prime }\right) \right) \\
&&-\left( q_{n,m}\left( \ell _{a}^{\prime }+\ell _{b}^{\prime }-m^{n}\right)
-2k_{n+1,m}\left( \ell _{a}+\ell _{b}\right) \right) \\
&&-2q_{n+1,m}\left( \ell _{b}\right) , \\
&\geq &2\left( k_{n+1,m}(\ell _{a})k_{n+1,m}(\ell _{b})+k_{n+1,m}(\ell
_{a})+k_{n+1,m}(\ell _{b})\right) +1 \\
&&-2q_{n,m}\left( \ell _{a}^{\prime }+\ell _{b}^{\prime }-m^{n}\right)
-2q_{n,m}\left( \ell _{a}^{\prime }\right) \\
&&+2k_{n+1,m}\left( \ell _{a}+\ell _{b}\right) \\
&&-2\left( k_{n+1,m}\left( \ell _{b}\right) +1\right) \text{ (by Remark 3)}
\\
&\geq &2\left( k_{n+1,m}(\ell _{a})k_{n+1,m}(\ell _{b})+k_{n+1,m}(\ell
_{a})+k_{n+1,m}(\ell _{b})\right) +1 \\
&&-2\left( q_{n,m}\left( \ell _{a}^{\prime }\right) +q_{n,m}\left( \ell
_{b}^{\prime }\right) -m\right) -2q_{n,m}\left( \ell _{a}^{\prime }\right) 
\text{ (by Lemma 7)} \\
&&+2\left( k_{n+1,m}\left( \ell _{a}\right) +k_{n+1,m}\left( \ell
_{b}\right) \right) \text{ (by Lemma 6)} \\
&&-2k_{n+1,m}\left( \ell _{b}\right) -2, \\
&\geq &2k_{n+1,m}(\ell _{a})k_{n+1,m}(\ell )+2m-1\text{ (by Cases 2112.ii \&
.iii)} \\
&>&0\text{ (since }m\geq 2\text{).}
\end{eqnarray*}

Case 2121: $\ell _{a}^{\prime }+\ell _{b}^{\prime }\geq m^{n}$, $%
q_{n,m}\left( \ell _{a}^{\prime }\right) \leq k_{n+1,m}(\ell _{a})$, $%
q_{n,m}\left( \ell _{b}^{\prime }\right) >k_{n+1,m}(\ell _{b})$, $%
q_{n,m}\left( \left( \ell _{a}+\ell _{b}\right) ^{\prime }\right) \leq
k_{n+1,m}(\ell _{a}+\ell _{b})$.

\qquad Then%
\begin{eqnarray*}
\Sigma _{n+1,m}\left( \ell _{a},\ell _{b}\right) &=&\text{I + II + III + IV}
\\
&\geq &2\left( k_{n+1,m}(\ell _{a})k_{n+1,m}(\ell _{b})+k_{n+1,m}(\ell
_{a})+k_{n+1,m}(\ell _{b})\right) -m+1 \\
&&+q_{n,m}\left( \ell _{b}^{\prime }\right) -q_{n,m}\left( \ell _{a}^{\prime
}+\ell _{b}^{\prime }-m^{n}\right) +m-q_{n,m}\left( \ell _{a}^{\prime
}\right) \\
&&+\left( -q_{n,m}\left( \ell _{a}^{\prime }\right) \right) +\left(
q_{n,m}\left( \ell _{b}^{\prime }\right) -2k_{n+1,m}\left( \ell _{b}\right)
\right) -\left( -q_{n,m}\left( \ell _{a}^{\prime }+\ell _{b}^{\prime
}-m^{n}\right) \right) \\
&&-2q_{n+1,m}\left( \ell _{b}\right) , \\
&\geq &2\left( k_{n+1,m}(\ell _{a})k_{n+1,m}(\ell _{b})+k_{n+1,m}(\ell
_{a})\right) +1 \\
&&+2q_{n,m}\left( \ell _{b}^{\prime }\right) -2q_{n,m}\left( \ell
_{a}^{\prime }\right) \\
&&-2\left( k_{n+1,m}\left( \ell _{b}\right) +1\right) \text{ (by Remark 3),}
\\
&\geq &2\left( k_{n+1,m}(\ell _{a})k_{n+1,m}(\ell _{b})+k_{n+1,m}(\ell
_{a})\right) -1 \\
&&+2\left( k_{n+1,m}\left( \ell _{b}\right) +1\right) -2k_{n+1,m}\left( \ell
_{a}\right) \text{ (by Cases 2121.ii \& 2121.iii)} \\
&&-2k_{n+1,m}\left( \ell _{b}\right) , \\
&=&2k_{n+1,m}(\ell _{a})k_{n+1,m}(\ell _{b})+1 \\
&>&0.
\end{eqnarray*}

\newpage

Case 2122: $\ell _{a}^{\prime }+\ell _{b}^{\prime }\geq m^{n}$, $%
q_{n,m}\left( \ell _{a}^{\prime }\right) \leq k_{n+1,m}(\ell _{a})$, $%
q_{n,m}\left( \ell _{b}^{\prime }\right) >k_{n+1,m}(\ell _{b})$, $%
q_{n,m}\left( \left( \ell _{a}+\ell _{b}\right) ^{\prime }\right)
>k_{n+1,m}(\left( \ell _{a}+\ell _{b}\right) )$.

\qquad Then%
\begin{eqnarray*}
\Sigma _{n+1,m}\left( \ell _{a},\ell _{b}\right) &=&\text{I + II + III + IV}
\\
&\geq &2\left( k_{n+1,m}(\ell _{a})k_{n+1,m}(\ell _{b})+k_{n+1,m}(\ell
_{a})+k_{n+1,m}(\ell _{b})\right) -m+1 \\
&&+q_{n,m}\left( \ell _{b}^{\prime }\right) -q_{n,m}\left( \ell _{a}^{\prime
}+\ell _{b}^{\prime }-m^{n}\right) +m-q_{n,m}\left( \ell _{a}^{\prime
}\right) \\
&&+\left( -q_{n,m}\left( \ell _{a}^{\prime }\right) \right) +\left(
q_{n,m}\left( \ell _{b}^{\prime }\right) -2k_{n+1,m}\left( \ell _{b}\right)
\right) \\
&&-\left( q_{n,m}\left( \ell _{a}^{\prime }+\ell _{b}^{\prime }-m^{n}\right)
-2k_{n+1,m}\left( \ell _{a}+\ell _{b}\right) \right) \\
&&-2q_{n+1,m}\left( \ell _{b}\right) , \\
&\geq &2\left( k_{n+1,m}(\ell _{a})k_{n+1,m}(\ell _{b})+k_{n+1,m}(\ell
_{a})\right) +1 \\
&&+2q_{n,m}\left( \ell _{b}^{\prime }\right) -2q_{n,m}\left( \ell
_{a}^{\prime }+\ell _{b}^{\prime }-m^{n}\right) -2q_{n,m}\left( \ell
_{a}^{\prime }\right) \\
&&+2k_{n+1,m}\left( \ell _{a}+\ell _{b}\right) \\
&&-2\left( k_{n+1,m}\left( \ell _{b}\right) +1\right) \text{ (by Remark 3),}
\\
&\geq &2\left( k_{n+1,m}(\ell _{a})k_{n+1,m}(\ell _{b})+k_{n+1,m}(\ell
_{a})\right) -1 \\
&&\text{ }+2q_{n,m}\left( \ell _{b}^{\prime }\right) -2\left( q_{n,m}\left(
\ell _{a}^{\prime }\right) +q_{n,m}\left( \ell _{b}^{\prime }\right)
-m\right) -2q_{n,m}\left( \ell _{a}^{\prime }\right) \text{ (by Lemma 7)} \\
&&+2\left( k_{n+1,m}(\ell _{a})+k_{n+1,m}(\ell _{b})\right) \text{ (by Lemma
6)} \\
&&-2k_{n+1,m}\left( \ell _{b}\right) \\
&\geq &2k_{n+1,m}(\ell _{a})k_{n+1,m}(\ell _{b})-1+2m\text{ (by Case 2122.ii)%
} \\
&>&0\text{. }
\end{eqnarray*}

\bigskip

Case 2211: $\ell _{a}^{\prime }+\ell _{b}^{\prime }\geq m^{n}$, $%
q_{n,m}\left( \ell _{a}^{\prime }\right) >k_{n+1,m}(\ell _{a})$, $%
q_{n,m}\left( \ell _{b}^{\prime }\right) \leq k_{n+1,m}(\ell _{b})$, $%
q_{n,m}\left( \left( \ell _{a}+\ell _{b}\right) ^{\prime }\right) \leq
k_{n+1,m}(\ell _{a}+\ell _{b})$.

\qquad Then%
\begin{eqnarray*}
\Sigma _{n+1,m}\left( \ell _{a},\ell _{b}\right) &=&\text{I + II + III + IV}
\\
&\geq &2\left( k_{n+1,m}(\ell _{a})k_{n+1,m}(\ell _{b})+k_{n+1,m}(\ell
_{a})+k_{n+1,m}(\ell _{b})\right) -m+1 \\
&&+q_{n,m}\left( \ell _{b}^{\prime }\right) -q_{n,m}\left( \ell _{a}^{\prime
}+\ell _{b}^{\prime }-m^{n}\right) +m-q_{n,m}\left( \ell _{a}^{\prime
}\right) \\
&&+\left( q_{n,m}\left( \ell _{a}^{\prime }\right) -2k_{n+1,m}\left( \ell
_{a}\right) \right) +\left( -q_{n,m}\left( \ell _{b}^{\prime }\right)
\right) -\left( -q_{n,m}\left( \ell _{a}^{\prime }+\ell _{b}^{\prime
}-m^{n}\right) \right) \\
&&-2q_{n+1,m}\left( \ell _{b}\right) , \\
&=&2\left( k_{n+1,m}(\ell _{a})k_{n+1,m}(\ell _{b})+k_{n+1,m}(\ell
_{b})\right) +1 \\
&&-2\left( k_{n+1,m}\left( \ell _{b}\right) +1\right) \text{ (by Remark 3),}
\\
&=&2k_{n+1,m}(\ell _{a})k_{n+1,m}(\ell _{b})-1 \\
&\geq &0\text{ unless }k_{n+1,m}(\ell _{b})=0\text{, which leads to a
contradiction as in } \\
&&\text{ \ \ \ \ \ \ \ \ \ Case 1111.}
\end{eqnarray*}%
\newpage

Case 2212: $\ell _{a}^{\prime }+\ell _{b}^{\prime }\geq m^{n}$, $%
q_{n,m}\left( \ell _{a}^{\prime }\right) >k_{n+1,m}(\ell _{a})$, $%
q_{n,m}\left( \ell _{b}^{\prime }\right) \leq k_{n+1,m}(\ell _{b})$, $%
q_{n,m}\left( \left( \ell _{a}+\ell _{b}\right) ^{\prime }\right)
>k_{n+1,m}(\ell _{a}+\ell _{b})$.

\qquad Then%
\begin{eqnarray*}
\Sigma _{n+1,m}\left( \ell _{a},\ell _{b}\right) &=&\text{I + II + III + IV}
\\
&\geq &2\left( k_{n+1,m}(\ell _{a})k_{n+1,m}(\ell _{b})+k_{n+1,m}(\ell
_{a})+k_{n+1,m}(\ell _{b})\right) -m+1 \\
&&+q_{n,m}\left( \ell _{b}^{\prime }\right) -q_{n,m}\left( \ell _{a}^{\prime
}+\ell _{b}^{\prime }-m^{n}\right) +m-q_{n,m}\left( \ell _{a}^{\prime
}\right) \\
&&+\left( q_{n,m}\left( \ell _{a}^{\prime }\right) -2k_{n+1,m}\left( \ell
_{a}\right) \right) +\left( -q_{n,m}\left( \ell _{b}^{\prime }\right) \right)
\\
&&-\left( q_{n,m}\left( \ell _{a}^{\prime }+\ell _{b}^{\prime }-m^{n}\right)
-2k_{n+1,m}(\ell _{a}+\ell _{b})\right) \\
&&-2q_{n+1,m}\left( \ell _{b}\right) , \\
&\geq &2\left( k_{n+1,m}(\ell _{a})k_{n+1,m}(\ell _{b})+k_{n+1,m}(\ell
_{b})\right) +1 \\
&&-2q_{n,m}\left( \ell _{a}^{\prime }+\ell _{b}^{\prime }-m^{n}\right) \\
&&+2k_{n+1,m}(\ell _{a}+\ell _{b}) \\
&&-2\left( k_{n+1,m}\left( \ell _{b}\right) +1\right) \text{ (by Remark 3),}
\\
&\geq &2k_{n+1,m}(\ell _{a})k_{n+1,m}(\ell _{b})-1 \\
&&-2\left( q_{n,m}\left( \ell _{a}^{\prime }\right) +q_{n,m}\left( \ell
_{b}^{\prime }\right) -m+1\right) \text{ (by Lemma 7)} \\
&&+2\left( k_{n+1,m}(\ell _{a})+k_{n+1,m}(\ell _{b})\right) \text{ (by Lemma
6),} \\
&\geq &2k_{n+1,m}(\ell _{a})\left( k_{n+1,m}(\ell _{b})+1\right) -1\text{
(by Case 2212.iii and } \\
&&\text{ \ \ \ \ \ \ \ Remark 5\ plus the fact that }\ell _{a}^{\prime }%
\text{ is less than }m^{n}\text{ } \\
&&\text{ \ \ \ \ \ \ \ \ so }q_{n,m}\left( \ell _{a}^{\prime }\right) \text{
is less than }m\text{),} \\
&\geq &0\text{ unless }k_{n+1,m}(\ell _{a})=0\text{ } \\
&\Rightarrow &k_{n+1,m}(\ell _{b})=0 \\
&\Rightarrow &\text{(by Case 2212.iii \& Remark 5) }\ell _{b}\left(
=k_{n+1,m}(\ell _{b})m^{n}+\ell _{b}^{\prime }\right) =0\text{, } \\
&&\text{ \ \ \ \ \ \ a contradiction.}
\end{eqnarray*}%
\newpage

Case 2221: $\ell _{a}^{\prime }+\ell _{b}^{\prime }\geq m^{n}$, $%
q_{n,m}\left( \ell _{a}^{\prime }\right) >k_{n+1,m}(\ell _{a})$, $%
q_{n,m}\left( \ell _{b}^{\prime }\right) >k_{n+1,m}(\ell _{b})$, $%
q_{n,m}\left( \left( \ell _{a}+\ell _{b}\right) ^{\prime }\right) \leq
k_{n+1,m}(\ell _{a}+\ell _{b})$.

\qquad Then%
\begin{eqnarray*}
\Sigma _{n+1,m}\left( \ell _{a},\ell _{b}\right) &=&\text{I + II + III + IV}
\\
&\geq &2\left( k_{n+1,m}(\ell _{a})k_{n+1,m}(\ell _{b})+k_{n+1,m}(\ell
_{a})+k_{n+1,m}(\ell _{b})\right) -m+1 \\
&&+q_{n,m}\left( \ell _{b}^{\prime }\right) -q_{n,m}\left( \ell _{a}^{\prime
}+\ell _{b}^{\prime }-m^{n}\right) +m-q_{n,m}\left( \ell _{a}^{\prime
}\right) \\
&&+\left( q_{n,m}\left( \ell _{a}^{\prime }\right) -2k_{n+1,m}\left( \ell
_{a}\right) \right) +\left( q_{n,m}\left( \ell _{b}^{\prime }\right)
-2k_{n+1,m}\left( \ell _{b}\right) \right) \\
&&-\left( -q_{n,m}\left( \ell _{a}^{\prime }+\ell _{b}^{\prime
}-m^{n}\right) \right) \\
&&-2q_{n+1,m}\left( \ell _{b}\right) , \\
&=&2k_{n+1,m}(\ell _{a})k_{n+1,m}(\ell _{b})+1 \\
&&+2q_{n,m}\left( \ell _{b}^{\prime }\right) \\
&&-2\left( k_{n+1,m}\left( \ell _{b}\right) +1\right) \text{ (by Remark 3),}
\\
&\geq &2k_{n+1,m}(\ell _{a})k_{n+1,m}(\ell _{b})+1 \\
&&+2\left( k_{n+1,m}(\ell _{b})+1\right) \text{ (by Case 2221.iii)} \\
&&-2\left( k_{n+1,m}\left( \ell _{b}\right) +1\right) \\
&=&2k_{n+1,m}(\ell _{a})k_{n+1,m}(\ell _{b})+1 \\
&>&0\text{. }
\end{eqnarray*}

Case 2222: $\ell _{a}^{\prime }+\ell _{b}^{\prime }\geq m^{n}$, $%
q_{n,m}\left( \ell _{a}^{\prime }\right) >k_{n+1,m}(\ell _{a})$, $%
q_{n,m}\left( \ell _{b}^{\prime }\right) >k_{n+1,m}(\ell _{b})$, $%
q_{n,m}\left( \left( \ell _{a}+\ell _{b}\right) ^{\prime }\right)
>k_{n+1,m}(\left( \ell _{a}+\ell _{b}\right) )$.

\qquad Then%
\begin{eqnarray*}
\Sigma _{n+1,m}\left( \ell _{a},\ell _{b}\right) &=&\text{I + II + III + IV}
\\
&\geq &2\left( k_{n+1,m}(\ell _{a})k_{n+1,m}(\ell _{b})+k_{n+1,m}(\ell
_{a})+k_{n+1,m}(\ell _{b})\right) -m+1 \\
&&+q_{n,m}\left( \ell _{b}^{\prime }\right) -q_{n,m}\left( \ell _{a}^{\prime
}+\ell _{b}^{\prime }-m^{n}\right) +m-q_{n,m}\left( \ell _{a}^{\prime
}\right) \\
&&+\left( q_{n,m}\left( \ell _{a}^{\prime }\right) -2k_{n+1,m}\left( \ell
_{a}\right) \right) +\left( q_{n,m}\left( \ell _{b}^{\prime }\right)
-2k_{n+1,m}\left( \ell _{b}\right) \right) \\
&&-\left( q_{n,m}\left( \ell _{a}^{\prime }+\ell _{b}^{\prime }-m^{n}\right)
-2k_{n+1,m}(\ell _{a}+\ell _{b})\right) \\
&&-2q_{n+1,m}\left( \ell _{b}\right) , \\
&\geq &2k_{n+1,m}(\ell _{a})k_{n+1,m}(\ell _{b})+1 \\
&&+2q_{n,m}\left( \ell _{b}^{\prime }\right) -2q_{n,m}\left( \ell
_{a}^{\prime }+\ell _{b}^{\prime }-m^{n}\right) \\
&&+2k_{n+1,m}(\ell _{a}+\ell _{b}) \\
&&-2\left( k_{n+1,m}\left( \ell _{b}\right) +1\right) \text{ (by Remark 3),}
\\
&\geq &2k_{n+1,m}(\ell _{a})k_{n+1,m}(\ell _{b})-1 \\
&&+2q_{n,m}\left( \ell _{b}^{\prime }\right) -2\left( q_{n,m}\left( \ell
_{a}^{\prime }\right) +2q_{n,m}\left( \ell _{b}^{\prime }\right) -m+1\right) 
\text{ (by Lemma 7)} \\
&&+2\left( k_{n+1,m}(\ell _{a})+k_{n+1,m}(\ell _{b})\right) \text{ (by Lemma
6)} \\
&&-2k_{n+1,m}\left( \ell _{b}\right) \\
&\geq &2k_{n+1,m}(\ell _{a})k_{n+1,m}(\ell _{b})-1 \\
&&+0\text{ (by Remark 5, }q_{n,m}\left( \ell _{a}^{\prime }\right) \leq m-1%
\text{)} \\
&&+2k_{n+1,m}(\ell _{a}) \\
&=&2k_{n+1,m}(\ell _{a})\left( k_{n+1,m}(\ell _{b})+1\right) -1 \\
&\geq &0\text{ unless }k_{n+1,m}(\ell _{a})=0\text{.}
\end{eqnarray*}

If $k_{n+1,m}(\ell _{a})=0$, then $k_{n+1,m}(\ell _{b})=0$, and by the
second inequality above,%
\begin{eqnarray*}
\Sigma _{n+1,m}\left( \ell _{a},\ell _{b}\right) &\geq &2k_{n+1,m}(\ell
_{a})k_{n+1,m}(\ell _{b})+1 \\
&&+2q_{n,m}\left( \ell _{b}^{\prime }\right) -2q_{n,m}\left( \ell
_{a}^{\prime }+\ell _{b}^{\prime }-m^{n}\right) \\
&&+2k_{n+1,m}(\ell _{a}+\ell _{b}) \\
&&-2\left( k_{n+1,m}\left( \ell _{b}\right) +1\right) \text{,} \\
&=&2\cdot 0+1 \\
&&+2q_{n,m}\left( \ell _{b}^{\prime }\right) -2q_{n,m}\left( \ell
_{a}^{\prime }+\ell _{b}^{\prime }-m^{n}\right) \\
&&+2k_{n+1,m}(\ell _{a}+\ell _{b}) \\
&&-2\left( 0+1\right) \text{,} \\
&\geq &2k_{n+1,m}(\ell _{a}+\ell _{b})-1\text{ } \\
&&\text{ \ \ \ \ (as above by Lemma 7 \& Remark 5),} \\
&>&0
\end{eqnarray*}%
since $k_{n+1,m}\left( \ell _{a}\right) =0\Rightarrow \ell _{a}=\ell
_{a}^{\prime }$, $k_{n+1,m}\left( \ell _{b}\right) =0\Rightarrow \ell
_{b}=\ell _{b}^{\prime },$ so $\ell _{a}+\ell _{b}=\ell _{a}^{\prime }+\ell
_{b}^{\prime }\geq $ $m^{n}$ (by Case 2222.i) and $k_{n+1,m}(\ell _{a}+\ell
_{b})\geq 1$.
\end{proof}

\section{Conclusions and Comments}

\begin{corollary}
$\forall \ell ,m,n\in \mathbb{N},$ $0\leq \ell \leq m^{n},$ $\left\vert
\Theta \right\vert \left( S(n,m);\ell \right) =$ $\left\vert \Theta \left(
Lex^{-1}\left( \ell \right) \right) \right\vert $ , \textit{i.e. the
(generalized and extended) Sierpinski graph, }$S(n,m)$, has Lex nested
solutions for the edge-isoperimetric problem.
\end{corollary}

\begin{proof}
This is the special case $s=0,t=m$ of Conjecture 1. It follows from Theorem
1 and the proof of Conjecture 2 in the Appendix.
\end{proof}

\begin{corollary}
Any $S(n,m)$ with external edges (see \cite{H-K-M-P}) has nested solutions
for $EIP$.
\end{corollary}

\begin{proof}
The vertices of $S(n,m)$ with external edges may be classified as in $I$, $J$
or $K$ and totally ordered so that $I<J<K$ without changing the structure.
The Corollary then follows from Theorem 1 for $S_{s,t}(n,m)$.
\end{proof}

\begin{corollary}
(Theorem of \cite{S-V}) The bisection width of $S(n,m)$ is given by the
formula%
\begin{equation*}
bw(S(n,m))=\left\{ 
\begin{tabular}{ll}
$\frac{m^{2}}{4}$ & if $m$ is even, \\ 
$n\left\lfloor \frac{m}{2}\right\rfloor ^{2}+\left\lfloor \frac{m}{2}%
\right\rfloor $ & if $m$ is odd%
\end{tabular}%
\text{.}\right. 
\end{equation*}
\end{corollary}

\begin{proof}
For any graph, $G$, $bw(G)=\left\vert \Theta \right\vert (G;\left\lfloor 
\frac{\left\vert V_{G}\right\vert }{2}\right\rfloor )$, so it follows from
Theorem 1 that $bw(S(n,m))=\left\vert \Theta (Lex^{-1}\left( \left\lfloor 
\frac{m^{n}}{2}\right\rfloor \right) \right\vert $. Savitha \& Vijayakumar 
\cite{S-V} give a beautiful proof of this corollary for $k$ even (and the
formula is easy to derive in that case). They also derive the correct
formula for $k$ odd, but their proof in that case is not adequate.
\end{proof}

\begin{corollary}
$\max \left\{ \left\vert \Theta \right\vert \left( S(n,m),\ell \right)
:0\leq \ell \leq m^{n}\right\} =n\left\lfloor \frac{m}{2}\right\rfloor
^{2}+\left\lfloor \frac{m}{2}\right\rfloor $.
\end{corollary}

\begin{corollary}
The Cheeger constant, $h(G)=\min \left\{ \frac{\left\vert \Theta \right\vert
\left( G;\ell \right) }{\ell }:\ell \leq \frac{\left\vert V_{G}\right\vert }{%
2}\right\} $ (See Wikipedia for background), of $S(n,m)$ is 
\begin{equation*}
h(S(n,m))=\left\{ 
\begin{tabular}{ll}
$\frac{m^{2}}{4}/\frac{m^{n}}{2}=\frac{1}{2m^{n-2}}$ & if $m$ is even, \\ 
$\frac{\left( m-1\right) }{2}\frac{\left( m+1\right) }{2}/\frac{\left(
m-1\right) m^{n-1}}{2}=\frac{m+1}{2m^{n-1}}$ & if $m$ is odd%
\end{tabular}%
\text{.}\right. 
\end{equation*}
\end{corollary}

\end{document}